\documentclass[a4paper,twoside,12pt]{amsart}

\usepackage{amsthm}
\usepackage{amscd}
\usepackage[CJKbookmarks=true]{hyperref}
\usepackage[center]{titlesec}
\usepackage{amscd}
\usepackage{amsfonts}
\usepackage{amsmath}
\usepackage{amssymb}
\usepackage{amsthm}
\usepackage{bm}
\usepackage{color} 
\usepackage{fancyhdr}
\usepackage{float}
\usepackage{fullpage}
\usepackage{graphics} 
\usepackage{graphicx}

\usepackage[CJKbookmarks=true]{hyperref}
\hypersetup{
	colorlinks=true,
	linkcolor=blue,
	filecolor=magenta, 
	urlcolor=cyan,
}
\usepackage{indentfirst}
\usepackage{latexsym}
\usepackage{mathrsfs}

\usepackage{pstricks}
\usepackage{tikz}
\usepackage[center]{titlesec}

\usepackage[all,poly,knot]{xy}

\titleformat{\section}{\centering\Large\bfseries}{\arabic{section}.}{1em}{}
\titleformat{\subsection}{\centering\large\bfseries}{\arabic{section}.\arabic{subsection}.}{1em}{}
\titleformat{\subsubsection}{\centering\large\bfseries}{\arabic{subsection}.\arabic{subsubsection}.}{1em}{}

\theoremstyle{plain}
\newtheorem{theorem}{Theorem}[section]
\newtheorem{thm}[theorem]{Theorem}

\newtheorem{prop}[theorem]{Proposition}
\newtheorem{corollary}[theorem]{Corollary}
\newtheorem{cor}[theorem]{Corollary}

\newtheorem{lem}[theorem]{Lemma}
\newtheorem{conjecture}[theorem]{Conjecture}

\theoremstyle{definition}

\newtheorem{defi}[theorem]{Definition}
\newtheorem{remark}[theorem]{Remark}
\newtheorem{rmk}[theorem]{Remark}

\newtheorem{question}[theorem]{Question}

\newtheorem{example}[theorem]{Example}

\numberwithin{equation}{theorem}

\newcommand{\frakR}{{\mathfrak R}}

\newcommand{\frakU}{{\mathfrak U}}

\newcommand{\frakX}{{\mathfrak X}}
\newcommand{\frakY}{{\mathfrak Y}}

\newcommand{\bA}{{\mathbb A}}
\newcommand{\bB}{{\mathbb B}}
\newcommand{\bC}{{\mathbb C}}
\newcommand{\bD}{{\mathbb D}}

\newcommand{\bL}{{\mathbb L}}
\newcommand{\bM}{{\mathbb M}}
\newcommand{\bN}{{\mathbb N}}

\newcommand{\bQ}{{\mathbb Q}}

\newcommand{\bV}{{\mathbb V}}
\newcommand{\bW}{{\mathbb W}}

\newcommand{\bZ}{{\mathbb Z}}

\newcommand{\mE}{{\mathcal E}}
\newcommand{\mF}{{\mathcal F}}
\newcommand{\mG}{{\mathcal G}}
\newcommand{\mH}{{\mathcal H}}

\newcommand{\mM}{{\mathcal M}}

\newcommand{\mO}{{\mathcal O}}

\newcommand{\mV}{{\mathcal V}}

\newcommand{\mX}{{\mathcal X}}
\newcommand{\mY}{{\mathcal Y}}

\newcommand{\nc}{\newcommand}
\nc{\on}{\operatorname}

\nc{\Aut} {{\on{\mathrm {Aut}}}}
\nc{\cech} {{\on{\check{\mathrm{C}}\mathrm{ech}}}}
\nc{\rmd} {{\on{\mathrm {d}}}}
\nc{\End} {{\on{\mathrm {End}}}}
\nc{\mEnd} {{\on{\mathcal{E}\mathrm{nd}}}} 
\nc{\Fil} {{\on{\mathrm {Fil}}}}
\nc{\Frac} {{\on{\mathrm {Frac}}}}
\nc{\Gal} {{\on{\mathrm {Gal}}}}
\nc{\GL} {{\on{\mathrm {GL}}}}
\nc{\gl} {{\on{\mathfrak {gl}}}}
\nc{\Gr} {{\on{\mathrm {Gr}}}}
\nc{\Hom} {{\on{\mathrm {Hom}}}}
\nc{\mHom} {{\on{\mathcal{H}\mathrm{om}}}} 
\nc{\id} {{\on{\mathrm {id}}}}
\nc{\Mor} {{\on{\mathrm {Mor}}}}
\nc{\pr} {{\on{\mathrm {pr}}}}
\nc{\PGL} {{\on{\mathrm {PGL}}}}
\nc{\rank} {{\on{\mathrm {rank}}}}
\nc{\Rep} {{\on{\mathrm {Rep}}}}
\nc{\Spec} {{\on{\mathrm {Spec}}}}
\nc{\Spa} {{\on{\mathrm {Spa}}}}
\nc{\Spf} {{\on{\mathrm {Spf}}}}

\titleformat{\section}{\centering\Large\bfseries}{\arabic{section}.}{1em}{}
\titleformat{\subsection}{\centering\large\bfseries}{\arabic{section}.\arabic{subsection}.}{1em}{}
\titleformat{\subsubsection}{\large\bfseries}{\arabic{section}.\arabic{subsection}.\arabic{subsubsection}.}{1em}{}

\newcommand{\para}[1]{{\ \\[5mm]\textbf{\emph{#1}}}}

\newcommand{\add}[1]{}
\newcommand{\addd}[1]{}

\def\hbV {{\widehat{\overline{\Ok}}}}
\def\hbk {{\widehat{\overline{k}}}}
\def\bark {{\overline k}}
\def\Ok {{\mO_k}}

\def\Ohbk{{\mO_{\widehat{\overline k}}}}

\def\Xhbk{{\frakX_{\hbk}}}
\def\Xhbkad{{\frakX_{\hbk}^{\rm ad}}}
\def\Xket{{(\frakX_{k})_{\text{\'et}}}}
\def\Xkadet{{(\frakX_{k}^{\mathrm{ad}})_{\text{\'et}}}}
\def\Xkad{{\frakX_k^{\mathrm{ad}}}}
\def\Xhbkadproet{{(\frakX_{\hbk}^{\mathrm{ad}})_{\text{pro\'et}}}}
\def\Xkadproet{{(\frakX_{k}^{\mathrm{ad}})_{\text{pro\'et}}}}
\def\Xhbkadet{{(\frakX_{\hbk}^{\mathrm{ad}})_{\text{\'et}}}}
\def\Xhbkadproet{{(\frakX_{\hbk}^{\mathrm{ad}})_{\text{pro\'et}}}}

\def\OfrakX{{\mO_{\frakX}}}
\def\OXp{{\mO_X^+}}
\def\hOX{{\hat\mO_X}}

\def\TT{{\underline{T}}}
\def\RK{{\frakR[\frac1p]}}

\def\barR{{\overline \frakR}}
\def\hR{{\widehat \frakR}}

\def\hbR{{\widehat{\overline \frakR}}}

\def\hbRK{{\widehat{\overline{\frakR}}[\frac1p]}}

\def\RbarV{{\frakR_{\mO_{\overline k}}}}

\def\hRbarV{{\frakR\hat{\otimes}{\mO_{\overline k}}}}
\def\hRbark{{\frakR\hat{\otimes}{\overline k}}}
\def\XhbV{{\frakX_{\mO_{\widehat{\overline k}}}}}
\def\XbarV{{\frakX_{\mO_{\overline k}}}}

\def\hbRKhbR{\left(\hbRK,\hbR\right)}

\def\Ainf{{A_{\rm inf}}} 
\def\BdRp{{B_{\rm dR}^+}} 
\def\BdR{{B_{\rm dR}}} 
\def\bBdRp{{\bB_{\rm dR}^+}} 
\def\OBdR{{\mO\bB_{\rm dR}}} 
\def\bAinfRh{{\bA_{\rm inf}\hbRKhbR}}
\def\OBinfRh{{\mO\bB_{\rm inf}\hbRKhbR}} 
\def\OC{{\mO\bC}}
\def\OCRh{{\OC\hbRKhbR}} 

\def\tOC{{\widetilde{\mO\bC}}}
\def\tOCRh{{\tOC\hbRKhbR}} 

\def\MF{{\mathcal{MF}}}

\def\RH{\mathcal{RH}}

\begin{document}

\title{A note on $p$-adic Simpson correspondence}
\author{Jinbang Yang}
\email{yjb@mail.ustc.edu.cn}
\address{School of Mathematical Sciences, University of Science and Technology of China, Hefei, Anhui 230026, PR China}
\author{Kang Zuo}
\email{zuok@uni-mainz.de}
\address{Institut f\"ur Mathematik, Universit\"at Mainz, Mainz 55099, Germany}
\date{}

\maketitle

\begin{abstract} Given a proper smooth $p$-adic variety, we show a comparison theorem for the $p$-adic Simpson correspondence constructed by Faltings and Riemann-Hilbert correspondence constructed by Scholze. As an application we formulate a sufficient condition for $\overline{\mathbb Q}_p$-local system being de Rham. We study a $p$-adic analogue of Simpson's $\bC^*$-action on the set of isomorphism classes of Higgs bundles and the corresponding Galois action on the set of isomorphism classes of generalized representations of the \'etale fundamental group.
\end{abstract}

\section{Introduction}

A major development in complex nonabelian Hodge theory was made by Hitchin~\cite{Hi} and Simpson~\cite{S1}. 
They constructed an equivalence between the category of irreducible representations of the fundamental group of a compact K\"ahler manifold and the category of stable Higgs bundles with vanishing Chern classes which generalizes the Donaldson-Uhlenbeck-Yau correspondence for stable vector bundles. 
As a $p$-adic analogue, Faltings~\cite{Fal05} constructed a so-called $p$-adic Simpson correspondence between the category of Higgs bundles and that of generalized representations of the \'etale fundamental group for a curve over a $p$-adic field. 
Faltings' construction was systematically studied by Abbes, Gros and Tsuji in~\cite{AGT}. 
Very recently, Scholze made exciting progress in $p$-adic Hodge theory. Together with Bhatt~\cite{BS15}, they introduced the pro\'etale site of an adic space as a refinement of the usual \'etale topology on the adic space. By using the sheaves of periods over the pro\'etale site, Scholze~\cite{Sch13} constructed a Riemann-Hilbert correspondence~\cite[Theorem 7.6]{Sch13}.

In this note we investigate the relationship between a modified Faltings' $p$-adic Simpson correspondence and Scholze's Riemann-Hilbert correspondence. Let $k$ be a complete discrete valuation field of mixed characteristic $(0,p)$ with perfect residue field $\kappa$ of characteristic $p>0$. Let $\frakX$ be a proper scheme over $\Spec(\Ok)$ with semistable reduction. Let 
\[\mE=(V,\nabla, \Fil)\] 
be a filtered de Rham bundle over the generic fiber $\frakX_k$ and 
\[(E,\theta):=(\Gr(V,\Fil),\Gr(\nabla)\cdot \xi^{-1})\]
the graded Higgs bundle attached to $(V,\nabla, \Fil)$, where 
\[\xi=p+\left[ \left(-p, (-p)^{\frac1p}, (-p)^{\frac1{p^2}} \cdots \right) \right]\]
is an element in Fontaine's period ring $\BdRp$.
Recall that Faltings' functor depends on the choice of an $A_2(\mO_k)$-lift of $\frakX$. 
The natural embedding 
\[k\rightarrow \BdRp\]
induces a lift of $\frakX_k$ over $\BdRp$. Under this non-integral lift, we slightly modify Faltings' original construction and get a functor $\bD^{Fal}$ mapping from the category of very small Higgs bundles to the category of generalized representations. 
Under this modified functor, one has that $(E,\theta)$ corresponds to a generalized geometric representation $\bD^{Fal}(E,\theta)$ of the geometric fundamental group $\pi_1(\frakX_{\bark})$.  
On the other hand, by Scholze's Riemann-Hilbert correspondence~\cite[Theorem 7.6]{Sch13} there is a $\bBdRp$-local system \[\mM^{Sch}(\mE) := \Fil^0( \nu^*\mE\otimes\OBdR )^{\nabla=0}\]
over the Scholze's pro\'etale site $\Xkadproet$. 
We show the following comparison theorem.
\begin{thm}[Theorem~\ref{thm:main2}]\label{001} There is an isomorphism of generalized representations
	\[\mM/\xi\mM \cong \bD^{Fal}(E,\theta),\]
	where $\mM$ is the restriction of $\mM^{Sch}(\mE)$ on ${\Xhbkadproet}$. 
\end{thm}

\begin{remark} We make the following two remarks.
\begin{itemize}
\item[i).] As a byproduct in the procession of proving Theorem~\ref{001}, we answer a question raised by Liu-Zhu about the relation between their $p$-adic Simpson correspondence and Faltings' $p$-adic Simpson correspondence (the first sentence in \cite[Section 2]{LZ}); see Corollary~\ref{Cor:LiuZhu}.
\item[ii).] Theorem~\ref{001} is a general form of a result proven by Tan and Tong on the comparison between de Rham and crystalline local systems in the following sense. In \cite[Proposition 3.22]{TT19}, they proved that a crystalline local system $\mathbb L^{\rm cris}$ is also a de Rham local system $\rho^{\rm dR}$. That is that the underlying de Rham bundle $(V,\nabla,\Fil)$ of the Fontaine-Faltings module corresponding the local system $\mathbb L^{\rm cris}$ is also the de Rham bundle corresponding the de Rham local system $\mathbb L^{\rm dR}$ constructed in \cite{Sch13}. It follows that the restriction of the crystalline local system $\rho^{\rm cris}=\rho^{\rm dR}$ to the geometric fundamental group corresponds to the Higgs bundle $(\Gr(V,\Fil),\Gr(\nabla)\cdot \xi^{-1})$ under Faltings' $p$-adic Simpson correspondence. 
\end{itemize}
\end{remark}

Theorem~\ref{001} has the following consequence. For a $\bZ_p$-local system $\bL$ over $\frakX_k$, there are two ways to construct Higgs bundles with trivial Chern classes. On one hand, we may assume $\bL$ to be very small after taking an \'etale base change on $\frakX_k$ and obtain the Higgs bundle 
\[(E,\theta)_\bL\]
corresponding to $\bL$ as a very small generalized geometric local system by the modified Faltings' functor. Note that $(E,\theta)_\bL$ has trivial Chern classes and is slope semistable with respect to any ample line bundle. On the other hand, Scholze defined a filtered $\mO_\Xket$-module with integrable connection 
\[\RH^{Sch}(\bL):=v_*(\OBdR\otimes\bL)\] attached to $\bL$. Subsequently, we introduce the associated $\mO_\Xket$-Higgs bundle \[(E,\theta)_{\bL}^{\mO\bB_{\rm dR}} := \Gr\circ\RH^{Sch}(\bL).\]
Indeed $(E,\theta)_{\bL}^{\mO\bB_{\rm dR}}$ has trivial Chern classes as $\RH^{Sch}(\bL)$ is a filtered de Rham bundle defined over a field of characteristic zero.
\begin{thm}[Theorem~\ref{maithm:subRep}]\label{thm:subRep} For a $\bZ_p$-local system $\bL$ over $\frakX_k$, the following statements hold.
\begin{itemize}
	\item[i).] There is an injective morphism between Higgs bundles
	\[(E,\theta)_\bL^{\mO\bB_{\rm dR}}\rightarrow (E,\theta)_\bL.\]
	Consequently, the Higgs bundle $(E,\theta)_\bL^{\mO\bB_{\rm dR}}$ is a semistable. We view $(E,\theta)_\bL^{\mO\bB_{\rm dR}}$ as a sub-Higgs bundle in $(E,\theta)_\bL$ via this injective morphism.
	\item[ii).] According Faltings' Simpson correspondence, the sub-Higgs bundle $(E,\theta)_\bL^{\mO\bB_{\rm dR}}$ corresponds to a sub-generalized representation. Then this generalized representation is actually a genuine sub-$\mO_{\hbk}$-representation, denoted by $\bW^{\rm geo}_{\Ohbk}$, of $\bL^{\rm geo}\otimes \mO_{\hbk}$. 
\end{itemize}
\end{thm}
\begin{remark}
Faltings conjectured that any semistable small Higgs bundle with trivial Chern classes corresponds to a $\hbk$-geometric representation. Theorem~\ref{thm:subRep}~(ii) is probably the easiest case in testing his conjecture as $(E,\theta)_{\bL}^{\mO\bB_{\rm dR}}$ is contained in the Higgs bundle $(E,\theta)_\bL$ arising from the $\hbk$-geometric representation $\bL^{\rm geo}\otimes \hbk$. The coefficient reduction from $\hat{\mO}_\frakX^+$ to $\mO_{\hbk}$ for the generalized representation corresponding to $(E,\theta)_{\bL}^{\mO\bB_{\rm dR}}$ actually comes from $\bL^{\rm geo}$ by using Faltings' functor of twisted pull-back of Higgs bundles and the property of twisted pull-back trivializable.
\end{remark}

\begin{corollary}\label{cor:001}
If $\bL$ is geometrically absolutely irreducible and $(E,\theta)_{\bL}^{\mO\bB_{\rm dR}}$ is nontrivial,
then $\bL$ is a de Rham representation.
\end{corollary}
\begin{rmk}
The result in Theorem~\ref{thm:subRep} and Corollary~\ref{cor:001} holds for $\overline{\bQ}_p$-local system, which can be reduced to the $\bZ_p$-local system case by descending the coefficient field and choosing a lattice in this descent. 
\end{rmk}

The action of Galois group $\Gal(\bark/k)$ on $\frakX_{\bark}$ induces a natural action on the category of generalized representations. In Section~\ref{sec:C^*}, by carefully checking the construction of Faltings' $p$-adic Simpson correspondence one finds that the corresponding action on the category of Higgs bundles can be basically described as the usual Galois action on the category of usual Higgs bundles with the action of the Higgs field twisted by a $1$-cocycle induced the element $\xi\in B_{dR}^+$. 
\begin{theorem}[Theorem~\ref{thm:rk2}]\label{Thm:rank2} The following statements hold.
\begin{itemize}
	\item[i).] A generalized representation corresponding to a graded Higgs bundle over $k$ is $\Gal(\bark/k)$-invariant and conversely, a Higgs bundle corresponding to a Galois-invariant generalized representation is nilpotent.
	\item[ii).] A rank-$2$ generalized representation is $\Gal(\bark/k)$-invariant if and only if the corresponding Higgs bundle is graded and defined over $k$ up to an isomorphism. 
\end{itemize}
\end{theorem}
We believe that ii) in Theorem~\ref{Thm:rank2} holds for representations of any rank.
\begin{conjecture}[Conjecture~\ref{conj:galHiggs}]
A generalized representation is $\Gal(\bark/k)$-invariant if and only if the corresponding Higgs bundle is graded and defined over $k$ up to an isomorphism. 
\end{conjecture}

\begin{remark} We make the following two remarks.
\begin{itemize}
	\item[i).] Shortly after submitting our note in Arxiv, Petrov~\cite{Pet20} arXiv:2012.13372 proved that a geometrically irreducible $\overline{\bQ}_p$-local system $\bL$ is de Rham after a twist by a character of $\text{Gal}(\bar k/k)$, which was conjectured by D. Litt. Though Corollary~\ref{cor:001}. in our note does not prove Litt's conjecture, it is closely related to Petrov's argument. The crucial points in Petrov's paper are:
	\begin{enumerate}
		\item globalizing Sen operator;
		\item using an eigenvalue of Sen operator to find a character $\chi$ such that $\RH^{Sch}(\bL\otimes\chi)\not=0$, it means equivalently $(E,\theta)^\OBdR_{\bL\otimes\chi}\not=0$;
		\item Using Diao, Lan, Liu and Zhu's result to get a comparison isomorphism between endomorphisms of $\bL$ and endomorphisms of Liu and Zhu's Higgs bundle $\mH(\bL)$, and the irreducibility of $\bL^\text{geo}$ implies 
		$\rank \RH^{Sch}(\bL\otimes\chi) =\rank (\bL\otimes\chi).$
	\end{enumerate}
	It is clear that once (2) is done, we may also apply Corollary~\ref{cor:001} to the non-trivial sub Higgs bundle $(E,\theta)^\OBdR_{\bL\otimes\chi}\subset (E,\theta)_{\bL\otimes\chi},$ which implies that $\bL\otimes\chi$ is de Rham.\\[.2cm]
	Note that Petrov actually only needs the assumption that
	$\text{End}(\bL|_{X_{\bar k}})=\bQ_p$,
	which is weaker than our assumption that $\bL^\mathrm{geo}$ is absolutely irreducible.
	\item[ii).] By applying i) in Theorem~\ref{Thm:rank2} to the graded Higgs bundle $(E,\theta)^\OBdR_{\bL}$,
	we see $\bW^{\rm geo}_{\mO_{\hbk}}$ is actually $\Gal(\bark/k)$-invariant. 
	In the previous version of our note we conjectured further that $\bW^{\rm geo}_{\mO_{\hbk}}$ will eventually 
	descend to a $\bZ_p$-\'etale subrepresentation
	$\bW\subset \bL. $
	However, it turns out that the conjecture is wrong.
	Thanks to Petrov, in an email exchange he provided us an example due to Berger \cite{Ber11} that neither the $\mO_{\hbk}$-coefficient of $\mathbb W^{\rm geo}_{\mO_{\hbk}}$ can be reduced to $\mO_{\bark}$ nor it can be lifted to a sub local system of $\bL$ over $X_{k'}$ for any finite extension of $k$.
\end{itemize} 
\end{remark}

{\bf Acknowledgment } We are grateful to Alexander Petrov for interesting discussions, in particular, the lifting property of the geometric sub local system constructed in Theorem~\ref{thm:subRep}.

\section{Preliminaries}~\label{section_preliminaries}

\subsection{Faltings' p-adic Simpson correspondence}
We recall from ~\cite{Fal05} the $p$-adic Simpson functor constructed by Faltings from the category of Higgs bundles to that of generalized representations of the \'etale fundamental group for liftable schemes over a $p$-adic field.\\[0cm]

Let $k$ be a complete discrete valuation field of mixed characteristic $(0,p)$ with perfect residue field $\kappa$ of characteristic $p>0$. Denote by $\Ok$ the ring of integers of $k$ and $\overline{k}$ the algebraic closure of $k$. For any scheme $\mY$ over $\mO_k$, denote by $\mY_k$ the generic fiber of $\mY$ and $\mY_{\kappa}$ the special fiber. Let $\frakX$ be a proper scheme over $\Ok$ with semistable reduction.

\subsubsection{The local correspondence}
Let $\frakU=\Spec(\frakR)$ be a small affine open subset of $\frakX$, that means $\frakR$ is \'etale over a toroidal model. By adjoining roots of characters of the torus, we obtain a subextension $\frakR_\infty$ of $\barR$ (the integral closure of $\frakR$ in the maximal \'etale cover of $\frakU_k$), which is Galois over \[\RbarV=\frakR\otimes_{\Ok}\mO_{\overline{k}}\]
with Galois group denoted by 
\[\Delta_\infty=\Gal(\frakR_\infty/\RbarV).\]
Let 
\[\Delta=\Gal(\barR/\RbarV)\]
and let $\hR$, $\hR_\infty$, $\hbR$, $\hRbarV$ be the $p$-adic completions of the corresponding rings.\\[.2cm]
We recall from \cite{Fal05} the notions of Higgs modules twisted by $\xi^{-1}$ and generalized representations, where $\xi$ is the element 
\[p + \left[\left(-p,(-p)^{\frac1p},(-p)^{\frac1{p^2}}\cdots \right)\right]\]
in Fontaine's period ring $\BdRp$.
\begin{defi}[{\cite[Section 3]{Fal05}}] Keep the notation as above.
\begin{itemize}
\item[i).] A \emph{Higgs module} is a pair $(E,\theta)$, where $E$ is a finitely generated free $\hRbarV$-module, and the Higgs field $\theta$ is an element in $\End_{\hRbarV}(E) \otimes_\frakR\xi^{-1}\Omega_{\frakR/{\Ok}}^1$ satisfies the integrable condition 
\[\theta\wedge\theta=0.\]
The Higgs module $(E,\theta)$ is called \emph{small} if $p^\alpha \mid \theta$ for some $\alpha>\frac1{p-1}$.
\item[ii).] A \emph{Higgs module on the generic fiber} is pair $(E,\theta)$, where $E$ is a finitely generated free $\hRbark$-module, and the Higgs field $\theta$ is contained in $\End_{\hRbark}(E) \otimes_\frakR\xi^{-1}\Omega_{\frakR/{\Ok}}^1$ satisfying the integrable condition. The Higgs module $(E,\theta)$ on the generic fiber is called \emph{small} if the $i$-th coefficient $\lambda_i$ of the characteristic polynomial of the Higgs field $\theta$ is contained in $\Omega_{\frakR/{\Ok}}^{\otimes i}\otimes (\hRbarV) \cdot p^{i\alpha}\xi^{-i}$ for all $i$ and some $\alpha>\frac2{p-1}$.
\end{itemize}
\end{defi}

\begin{defi}[{\cite[Section 3]{Fal05}}] Keep the notation as above.
\begin{itemize}
\item[i).] A \emph{generalized representation} of $\Delta$ is a pair $(\bV,\rho)$, where $\bV$ is a finitely generated free $\hbR$-module, and $\rho$ is a continuous semilinear action of $\Delta$ on $\bV$. The generalized representation $(\bV,\rho)$ is called \emph{small} if there exists a $\Delta$-invariant basis of the free $\hbR/p^{2\alpha}$-module $\bV/p^{2\alpha}$.
\item[ii).] A \emph{generalized $\bQ_p$-representation} of $\Delta$ is a pair $(\bV,\rho)$, where $\bV$ is a finitely generated free $\hbRK$-module, and $\rho$ is a continuous semilinear action of $\Delta$ on $\bV$. The generalized $\bQ_p$-representation $(\bV,\rho)$ is called \emph{small} if it comes from a small generalized representation by tensoring $\bQ_p$.
\end{itemize}
\end{defi}

\begin{thm}[{\cite[Theorem 3]{Fal05}} Local version of $p$-adic Simpson correspondence] \label{Fal_local} For a small affine space $\frakU$, there exist an equivalent functor $\bD^{Fal}_{\frakU}$ mapping from the category of small Higgs bundles to the category of small generalized representations and an equivalent functor, still denoted by $\bD^{Fal}_{\frakU}$, mapping from the category of small Higgs bundles on the generic fiber to the category of small generalized $\bQ_p$-representations.
\end{thm}

The construction of the second functor in the above Theorem is similar to that of the first one. In the following we only recall from \cite[p.850-851]{Fal05} the construction of the first equivalent functor $\bD^{Fal}_{\frakU}$ sending a Higgs bundle $(E,\theta)$ to a generalized representation 
\[(\bV,\rho)=\bD^{Fal}_{\frakU}(E,\theta).\]
The underlying module of the generalized representation is defined to be 
\begin{equation} \label{underlyingMod}
	\bV := E\otimes_{\hRbarV} \hbR.
\end{equation}
The semilinear action $\rho$ is a little more complicated. Before giving its construction, some notation $\widetilde{\frakR},\tau,\TT,\widetilde{\TT}$ is needed to introduce as follows. Let
\[\frakU=\Spec(\frakR)\]
be a small affine space over $\Ok$. Firstly, let $\widetilde{\frakR}$ be a chosen smooth algebra over 
\[A_2({\Ok}):= \Ainf/\xi^2\Ainf\]
such that it is a lift of $\hRbarV$ over $A_2(\Ok)$ (that means $\hRbarV\cong\widetilde{\frakR}\otimes_{A_2(\Ok)} \mO_\bark$), where 
\[\Ainf= W(\varprojlim_{\Phi} \mO_{\hbk}/p)\]
is an Fontaine's period ring; see \cite[Sections 2.1 and 2.2]{Fon82} for more details. Note that any two such lifts are isomorphic but not canonically isomorphic. Secondly, let 
\[\tau \colon \widetilde{\frakR}\rightarrow A_2(\frakR):=\bAinfRh/\xi^2\bAinfRh\] 
be a chosen lift of the inclusion 
\[\frakR\hookrightarrow \hbR\]
over $A_2({\Ok})$, where
\[\bAinfRh=W\left(\varprojlim\limits_{\Phi} \barR/p\right)\]
is a Witt ring with coefficients in $\varprojlim\limits_{\Phi} \barR/p$, see \cite[p.28]{Fal89} or \cite[p.35]{Sch13} for more details. Ultimately, let 
\[\TT=\{T_1,\cdots,T_d\}\]
be chosen parameters of $\frakR$ over $\Ok$ and let 
\[\widetilde{\TT}=\{\widetilde{T}_1,\cdots,\widetilde{T}_d\}\]
be chosen lifts of $\TT$ in $\widetilde{\frakR}$, where $d$ is the relative dimension of $\frakR$ over $\Ok$. We have the following commutative diagram
\begin{equation}
\xymatrix{
A_2({\Ok}) \ar[r] \ar@{->>}[d] & \widetilde{\frakR} \ar[r]^{\tau} \ar@{->>}[d] & A_2(\frakR) \ar@{->>}[d] \\
\hbV \ar[r] & \hRbarV \ar[r] & \hbR\\
}
\end{equation}

Using the fact that the Higgs field $\theta$ behaves like a connection, for any two lifts 
\[\tau,\tau'\colon \widetilde{\frakR} \rightarrow A_2(\frakR),\]
Faltings~\cite[page 851]{Fal05} constructed an $\hbR$-linear isomorphism 
\[\beta_{\tau,\tau'} \colon \bV \rightarrow \bV,\]
defined as, for any $m\in E$,
\begin{equation}\label{equ:taylorHiggs}
\beta_{\tau,\tau'}(m\otimes1) = \sum_{I=(i_1,\cdots,i_d)\in \bN^d} \theta(\xi\partial)^I(m)/I! \otimes \left(\frac{\tau(\widetilde{T}) - \tau'(\widetilde{T})}{\xi}\right)^I,
\end{equation}
where 
\[\theta(\xi\partial)^I=\prod_{\ell=1}^{d}\theta(\xi\partial_\ell)^{i_{\ell}}, \quad I!=\prod_{\ell=1}^{d}i_{\ell}!,\]
\[\left(\frac{\tau(\widetilde{T}) - \tau'(\widetilde{T})}{\xi}\right)^I=\prod_{\ell=1}^{d}\left(\frac{\tau(\widetilde{T_\ell}) - \tau'(\widetilde{T_\ell})}{\xi}\right)^{i_{\ell}},\]
and
 \[\partial_i=\partial/\partial T_i\]
is the dual base of $\frakR$-derivations. For any $\sigma\in \Delta$, it induces an automorphism of $A_2(\frakR)$ in a natural way. Then the composition 
\[\sigma\circ \tau\colon \widetilde{\frakR}\rightarrow A_2(\frakR)\]
is another lift of the inclusion
\[\frakR\hookrightarrow \hbR\]
over $A_2({\Ok})$. One has
\begin{equation}
\beta_{\sigma\circ\tau,\tau}(m\otimes1) = \sum_{I\in \bN^d} \theta(\xi\partial)^I(m)/I! \otimes \left(\frac{\tau(\widetilde{T})^{\sigma} - \tau(\widetilde{T})}{\xi}\right)^I.
\end{equation}
Then the composition 
\[\rho(\sigma):=\beta_{\sigma\circ\tau,\tau}\circ(1\otimes \sigma) \colon \bV\rightarrow \bV\]
is a $\sigma$-semilinear map over $\hbR$. The semilinear action $\rho$ of $\Delta$ on the free $\hbR$-module $\bV$ is defined to be sending $\sigma$ to $\rho(\sigma)$. Note that $\rho$ is independent of the choice of those data $(\widetilde{\frakR},\tau,\TT,\widetilde{\TT})$ up to canonical isomorphism. In other words, there exists a canonical isomorphism 
\begin{equation} \label{gluing isomorphisms}
(\bV,\rho_{(\widetilde{\frakR},\tau,\TT,\widetilde{\TT})}) \cong (\bV,\rho_{(\widetilde{\frakR}',\tau',\TT',\widetilde{\TT}')})
\end{equation}
for two different choices $(\widetilde{\frakR},\tau,\TT,\widetilde{\TT})$ and $(\widetilde{\frakR}',\tau',\TT',\widetilde{\TT}')$.

\subsubsection{The global correspondence}\label{section_global_corr}
In this subsection, we recall from \cite[Section 4]{Fal05} the globalization of Higgs bundles and generalized representations.

Let $\frakX$ be a proper scheme over $\Spec({\Ok})$ with semistable reduction. A \emph{generalized} ($\bQ_p$-) \emph{representation} of the geometric fundamental group $\pi_1(\frakX_\bark)$ of generic fiber of $\frakX$ is defined to be a compatible system of generalized ($\bQ_p$-)representations on a covering of $\frakX$ by small open affine subsets. It is called \emph{small}, if the generalized representations in the system are small. A \emph{Higgs bundle} over $\frakX_\Ohbk$ (resp. over $\frakX_\hbk$) is a vector bundle $E$ on $\frakX_\Ohbk$ (resp. on $\frakX_\hbk$) together with a Higgs field $\theta$ in $\End(E)\otimes\xi^{-1}\Omega^1_{\frakX/{\Ok}}$. A Higgs bundle $(E,\theta)$ over $\frakX_\Ohbk$ is called \emph{small} if \[p^\alpha\mid \theta \text{, for some } \alpha>\frac1{p-1}\]
 and a Higgs bundle $(E,\theta)$ over $\frakX_\hbk$ is called \emph{small} if the $i$-th coefficient $\lambda_i$ of the characteristic polynomial of the Higgs field $\theta$ is contained in $p^{i\alpha}\xi^{-i}\Omega_{\frakX_\Ohbk/{\Ohbk}}^{\otimes i}$ for all $i=1,2,\cdots,d$ and some $\alpha>\frac2{p-1}$, where $d$ is the relative dimension of $\frakX$ over $\Ok$.\\[0mm]

By gluing the local equivalences in Theorem~\ref{Fal_local}, global equivalences of categories are obtained.
\begin{thm}[{\cite[Theorem 5]{Fal05}} Global version of $p$-adic Simpson correspondence]\label{thm:Flatings}
Assume that $\frakX$ is liftable to $A_2(\Ok)$. Then there exist an equivalence mapping from the category of small Higgs bundles over $\frakX_\Ohbk$ to that of small generalized representations of $\pi_1(\frakX_\bark)$, and an equivalence mapping from the category of small Higgs bundles over $\frakX_\hbk$ to that of small generalized $\bQ_p$-representations of $\pi_1(\frakX_\bark)$.
\end{thm}

Recall from \cite[p.855]{Fal05} the construction of the first global equivalent functor sending a Higgs bundle $(E,\theta)$ to a generalized representation $(\bV,\rho)$ briefly. Let $\widetilde{\frakX}$ be a chosen lift of $\frakX$ over $A_2({\Ok})$. Then it is a first order thickening of $\frakX_{\mO_\bark}$, and so its underlying topological space is the same as that of $\frakX_{\mO_\bark}$. For any covering $\{\frakU_i\}_i$ of $\frakX$ by small open affine subsets \[\frakU_i=\Spec(\frakR_i),\]
one also gets a covering $\{\widetilde{\frakU}_i\}_i$ of $\widetilde{\frakX}$ by open affine subsets \[\widetilde{\frakU}_i=\Spec(\widetilde{\frakR}_i)\]
satisfying 
\[\frakR_{i,\mO_\bark}=\widetilde{\frakR}_i \otimes_{A_2(\Ok)} \mO_\bark.\]
For each $i$, fix an embedding
\[\tau_i\colon \widetilde{\frakR}_i= \mO_{\widetilde{\frakX}} (\widetilde{\frakU}_i) \hookrightarrow A_2(\frakR_i),\]
local parameters $\TT_i$ and $\widetilde{\TT}_i$. Recall from Theorem~\ref{Fal_local} that for each $\frakU_i$, there is an equivalent functor $\bD_{\frakU_i}^{Fal}$ mapping from the category of small Higgs bundles to that of small generalized representations
\[(E,\theta) \mapsto \bD^{Fal}_{\frakU_i}(E,\theta):=\left(E(\frakU_i)\otimes_{\frakR_i} \hbR_i,\rho_{(\widetilde{\frakR}_i,\tau_i,\TT_i,\widetilde{\TT}_i)}\right).\]
For a small Higgs bundle $(E,\theta)$ over $\XbarV$, one gets a family of local small generalized representations:
\[\Big(\bD^{Fal}_{\frakU_i}(E\mid_{\frakU_i},\theta\mid_{\frakU_i})\Big)_i := \left(E(\frakU_i)\otimes_{\frakR_i}\hbR_i,\rho_{(\widetilde{\frakR}_i,\tau_i,\TT_i,\widetilde{\TT}_i)}\right)_i.\]
These local generalized representations are glued into a global small generalized representation via the isomorphisms in \eqref{gluing isomorphisms} over $\frakU_{i}\cap\frakU_j$, which is denoted by
\begin{equation}\label{Fal_functor}
\bD^{Fal}_{\widetilde{\frakX}}(E,\theta).
\end{equation}

\begin{remark}[{\cite[p.855]{Fal05}}]\label{rmk_diff_lift} Note that the above global equivalence $\bD^{Fal}_{\widetilde{\frakX}}$ depends on the choice of such a lift $\widetilde{\frakX}$. For another lift $\widetilde{\frakX}'$, since the set of all lifts forms a $H^1(\XhbV,T_{\XhbV}^1\cdot\xi)$-torsor, one has 
\[c:= \widetilde{\frakX}' -\widetilde{\frakX} \in H^1(\XhbV,T_{\XhbV}^1\cdot\xi).\]
There is a nature commutative diagram 
\begin{equation*}
\xymatrix@R=2cm{
\{\text{small Higgs bundles over $\frakX_\Ohbk$}\} \ar[r]^-{\bD^{Fal}_{\widetilde{\frakX}}} \ar[d]^{\text{twisted by } c} & \{\txt{small generalized representations of $\pi_1(\frakX_\bark)$}\} \ar@{=}[d]\\ 
\{\text{small Higgs bundles over $\frakX_\Ohbk$}\} \ar[r]^-{\bD^{Fal}_{\widetilde{\frakX}'}} & \{\text{small generalized representations of $\pi_1(\frakX_\bark)$}\} \\ 
}
\end{equation*} 
\end{remark}

\subsubsection{Modification of Faltings' construction.} In this subsection, we modify Faltings' functor $\bD^{Fal}_{\widetilde{\frakX}}$ in Section~\ref{section_global_corr} in order to compare this functor with the functor defined by Scholze in~\cite[Theorem 7.6]{Sch13}. Note that Faltings' functor depends on the choice of the lift $\widetilde{\frakX}$, while Scholze's functor does not. Hence we need to find a canonical lift for Faltings functor. Actually there does exist some canonical one, but it is not integral in general. This is why we need to modify Faltings' functor. The idea is already given in the last two sentences in the second paragraph in \cite[Section 4]{Fal05}.\\[0mm]

Let $k$ be a $p$-adic field with maximal ideal of the ring of integers generated by an element $\pi$ and with residue field $\kappa$. Recall from \cite[Section 2.8]{Fon82} Fontaine's period ring 
\[\BdRp:= \varprojlim_{n\in \bN} W(R)[\frac1p]/(\xi^n),\]
the maximal unramified subfield $W(\kappa)[\frac1p]$ in $k$ is naturally contained in the period ring $\BdRp$. Since $\BdRp$ is a complete discrete valuation ring with residue field $\hbk$. By \cite[Theorem 0.21]{FO}, the element $\pi$ can be uniquely lifted to an element in the period ring $\BdRp$, which has the same minimal polynomial as $\pi$ over $W(\kappa)[\frac1p]$. Thus the field 
\[k=W(\kappa)[\frac1p][\pi]\]
is naturally embedded into $\BdRp$ by sending $\pi$ to that lift. The composition 
\[k\hookrightarrow\BdRp\twoheadrightarrow W(R)[\frac1p]/(\xi^2)=A_2({\Ok})[\frac1p]\] 
is still an embedding since $k$ is a field. Denote by $\alpha_0$ the smallest nonnegative rational number such that $\pi \in p^{-\alpha_0} A_2(\Ok)$. Then one has 
\[A_2(\Ok)[\pi] = A_2(\Ok) + \Ohbk\cdot\frac{\xi}{p^{\alpha_0}}.\]
Denote the $p$-adic completion of $\frakR\otimes_{\Ok} A_2({\Ok})[\pi]$ by
\[\widetilde{\frakR}:= \frakR\hat{\otimes}_{\Ok} A_2({\Ok})[\pi].\]
It is a lift of $\hRbarV$ over $A_2(\Ok)[\pi]$, which will be called the \emph{canonical lift} over $A_2(\Ok)[\pi]$. Let \[\TT=\{T_1,\cdots,T_d\}\]
be some chosen parameters of $\frakR$ and 
\[\tau\colon \widetilde{\frakR} \rightarrow A_2(\frakR)[\pi]\]
be a chosen lift of the inclusion $\frakR\hookrightarrow \hbR$ over $A_2(\Ok)[\pi]$. We identify $\frakR$ with subring $\frakR\otimes1$ in $\widetilde{\frakR}$. It is clear that 
\[\tau(\frakR) \subset A_2(\frakR) +\hbR\cdot\frac{\xi}{p^{\alpha_0}}.\]
Note that elements in $\TT$ are also parameters of $\widetilde{\frakR}$ over $A_2(\Ok)[\pi]$. Thus the lift $\tau$ is uniquely determined by its restriction on $\frakR$. By abuse of notation, the restriction of $\tau$ on $\frakR$ is still denoted by $\tau$
\begin{equation}\label{equ:lift_tau}
\tau\colon \frakR \rightarrow A_2(\frakR)[\pi].
\end{equation}
For another lift
\[\tau'\colon \widetilde{\frakR} \rightarrow A_2(\frakR)[\pi],\]
one has
\[\tau(T_i) - \tau'(T_i) \in \hbR\cdot\frac{\xi}{p^{\alpha_0}}.\]

For a small Higgs bundle $(E,\theta)$, the series in \eqref{equ:taylorHiggs} may no longer converge. But if we require the Higgs field $\theta$ to be sufficiently small, then the series in \eqref{equ:taylorHiggs} is still convergent. This leads us to give the following definition.
\begin{defi}\label{def:verysmall}
	A Higgs bundle $(E,\theta)$ is called \emph{very small} if 
	\[p^{\alpha}\mid\theta \text{, for some } \alpha>\alpha_0+\frac{1}{p-1}.\]
	A Higgs bundle $(E,\theta)$ on the generic fiber is called \emph{very small} if the $i$-th coefficient $\lambda_i$ of the characteristic polynomial of the Higgs field $\theta$ is contained in $\Omega_{\frakX_\Ohbk/{\Ohbk}}^{\otimes i}\otimes p^{i\alpha}\xi^{-i}$ for all $i$ and some $\alpha>\alpha_0+\frac2{p-1}$.
\end{defi}

Let 
\[\tau,\tau'\colon \widetilde{\frakR} \rightarrow A_2(\frakR)[\pi]\]
be two lifts of the inclusion 
\[\frakR\hookrightarrow \hbR\]
over $A_2(\Ok)[\pi]$. Similar as \eqref{equ:taylorHiggs}, we define an $\hbR$-linear isomorphism 
\[\beta_{\tau,\tau'}\colon \bV\rightarrow \bV\]
as follows
\begin{equation}\label{modified_iso}
\beta_{\tau,\tau'}(m\otimes1) := \sum_{I\in\bN^d} \frac{\theta(p^{-\alpha_0}\xi\partial)^I(m)}{I!}\otimes\left(\frac{\tau(T)-\tau'(T)}{p^{-\alpha_0}\xi}\right)^I.
\end{equation}
For any $\sigma\in \Delta$, we set
\begin{equation}\label{equ:galoisAction}
\rho_\tau(\sigma):=\beta_{\sigma\circ\tau,\tau}\circ(1\otimes \sigma).
\end{equation}
Then we get a generalized representation $(\bV,\rho_\tau)$. The isomorphism in \eqref{modified_iso} is an isomorphism of generalized representations
\[\beta_{\tau,\tau'} \colon (\bV,\rho_\tau) \xrightarrow{\sim} (\bV,\rho_{\tau'}),\]
 which means that the generalized representation $(\bV,\rho_\tau)$ is independent of the choice of the lift $\tau$ up to isomorphism. Thus one gets a well-defined local modified functor sending a small Higgs bundles $(E,\theta)$ to the corresponding generalized representation $(\bV,\rho_\tau)$. For a small open affine covering of $\frakX$, the local functors are glued into a global functor $\bD^{Fal}$, which will be called modified Faltings' functor.
\begin{thm}[modified global version of $p$-adic Simpson correspondence]\label{thm:modifiedFunctor}
Keep notation as above. The functor $\bD^{Fal}$ mapping from the category of very small Higgs bundles over $\frakX_\Ohbk$ to that of generalized representations of $\pi_1(\frakX_\bark)$ is fully faithful. This functor also extends to the $\bQ_p$-theory. That is, there exists a fully faithful functor, still denoted by $\bD^{Fal}$, mapping from the category of very small Higgs bundles over $\frakX_\hbk$ to that of generalized $\bQ_p$-representations of $\pi_1(\frakX_\bark)$.
\end{thm}

\begin{proof} Let $\widetilde{\frakX}'$ be a chosen lift of $\frakX$ over $A_2(\Ok)$ and $\bD^{Fal}_{\widetilde{\frakX}'}$ be the Faltings' functor in~\eqref{Fal_functor}. Then by tensoring with $A_2(\Ok)[\pi]$, we get a lift $\widetilde{\frakX}'\otimes_{A_2(\Ok)} A_2(\Ok)[\pi]$ of $\frakX$ over $A_2(\Ok)[\pi]$. Let 
\[\widetilde{\frakX} = \frakX\hat\otimes_{\Ok} A_2(\Ok)[\pi]\]
be the canonical lift of $\frakX$ over $A_2(\Ok)[\pi]$. As in Remark~\ref{rmk_diff_lift}, the difference between $\widetilde{\frakX}$ and $\widetilde{\frakX}'\otimes_{A_2(\Ok)} A_2(\Ok)[\pi]$ is a $1$-Cech-cocycle in $H^1(\XhbV,T_{\XhbV}^1\cdot p^{-\alpha_0}\xi)$. Twisted a Higgs bundle by this cocycle, one constructs an equivalent self-functor of the category of very small Higgs bundles. The similar argument as in \cite[page 855]{Fal05} implies that the functor $\bD^{Fal}_{\widetilde{\frakX}'}$ and the modified functor $\bD^{Fal}$ fit into the following commutative diagram
	\begin{equation*}
	\xymatrix@C=2cm{
		\{\text{very small Higgs bundles}\} \ar[d]^{\text{twisted by cocycle}} \ar[r]^-{\bD^{Fal}}& \{\text{generalized representations}\} \ar@{=}[d] \\
		\{\text{very small Higgs bundles}\} \ar[r]^-{\bD^{Fal}_{\widetilde{\frakX}'}} & \{\text{generalized representations}\} \\
	}
	\end{equation*}
	Since $\bD^{Fal}_{\widetilde{\frakX}'}$ is an equivalent functor between the category of small Higgs bundles and that of small generalized representations, the functor $\bD^{Fal}$ is fully faithful.
\end{proof}

\begin{defi}\label{def:FalFun} The essential images of $\bD^{Fal}$ are called \emph{very small generalized} ($\bQ_p$-) \emph{representations}. Thus $\bD^{Fal}$ is an equivalent functor mapping from the category of very small Higgs bundles to that of very small generalized representations.
\begin{equation*}
	\xymatrix@C=1cm{
		\{\text{very small Higgs bundles}\} \ar@/^2pt/[r]^-{\bD^{Fal}} & \{\text{very small generalized representations}\} \ar@/^2pt/[l]^-{\mH^{Fal}} \\ 
	}
\end{equation*}
where $\mH^{Fal}$ denotes the quasi-inverse functor of $\bD^{Fal}$.
\end{defi}

\subsection{Description of generalized representations via sheaves of periodic rings}\label{section_discrip_via_Sheaves}

In this subsection, we recall some notation and facts for later.\\[0mm]

Let $\frakX$ be a proper scheme over $\Ok$ with semistable reduction. Denote by $\frakX_k$ the smooth generic fiber and denote by $\Xkad$ the corresponding adic space. Let $\Xkadproet$ be the pro\'etale site~\cite[Definition 3.9]{Sch13} of $\Xkad$. 
There is a natural projection 
\[\nu\colon \Xkadproet \rightarrow \Xkadet\]
mapping from this pro\'etale site to the \'etale site of $\Xkad$. Since our basis space $\frakX$ is proper, according~\cite[Theorem 9.1]{Sch13} or K\"opf's thesis~\cite{Kopf}, There is an equivalence between the category of the coherent sheaves over $\frakX_k$ and that of the coherent sheaves over $\Xkad$. By abusing of notation, we will use a single notation standard for a coherent sheaf over $\frakX_k$ and its corresponding sheaf over $\Xkad$. For example, once one gets a vector bundle $\mE$ over $\Xkad$ we will also denote by $\mE$ the corresponding vector bundle over $\frakX_k$ and vice versa.\\[0mm] 

We Recall from~\cite{Sch13} some sheaves over $\Xkadproet$. By pulling back the structure sheaf $\mO_{\Xkadet}$ via this projection $\nu$, one obtains the \emph{uncompleted structure sheaf} 
\[\mO_{\Xkad}=\nu^*\mO_{\Xkadet}\] 
and the \emph{integral uncompleted structure sheaf} 
\[\mO^+_{\Xkad}=\nu^*\mO^+_{\Xkadet}\]
of the pro\'etale site $\Xkadproet$~\cite[Definition 4.1(i)]{Sch13}.
Taking the $p$-adic completion, one obtains the \emph{integral completed structure sheaf} 
\[\hat\mO^+_{\Xkad}=\varprojlim_n \mO^+_{\Xkad}/p^n\]
and the \emph{completed structure sheaf} 
\[\hat\mO_{\Xkad} = \hat\mO^+_{\Xkad}[\frac1p]\]
of the pro\'etale site $\Xkadproet$~\cite[Definition 4.1(ii)]{Sch13}. \\

Let $\Xhbkadproet$ be the pro\'etale site of $\Xhbkad$. One also has a natural projections 
\[\nu'\colon \Xhbkadproet \rightarrow \Xhbkadet\]
and some structure sheaves over this pro\'etale site $\Xhbkadproet$.
The map 
\[\Xhbk\rightarrow \frakX_k\]
induces natural morphisms 
\[\Xhbkadproet \rightarrow \Xkadproet\]
and 
\[\Xhbkadet\rightarrow\Xkadet.\]
These morphisms fit in the following commutative diagram
\begin{equation}\label{diag_cano_map}
\xymatrix{
	\Xhbkadproet\ar[r]^{\nu'} \ar[d] & \Xhbkadet \ar[d] \\
	\Xkadproet \ar[r]^{\nu} & \Xkadet,\\
}
\end{equation}
The structure sheaves of $\Xhbkadproet$ and of $\Xkadproet$ are compatible. That is 
\[\hat{\mO}_{\Xkad}^+\mid_{\Xhbkadproet} = \hat{\mO}_{\Xhbkad}^+ \text{ and } \hat{\mO}_{\Xkad}\mid_{\Xhbkadproet} = \hat{\mO}_{\Xhbkad}.\]
\\[0mm]

For any small open affine subset 
\[\frakU =\Spec(\frakR)\subset \frakX,\]
denote
\begin{equation}\label{universal-covering}
\overline{\frakU} = \varprojlim \limits_{V/\frakU_{\hbk}^{\mathrm{ad}}: \text{finite \'etale}} V,
\end{equation}
which is an object in the site $\Xhbkadproet$. From the definitions of the complete structure sheaves, one has 
\[\hat{\mO}_{\Xhbkad}^+(\overline{\frakU}) =\hbR \text{ and } \hat{\mO}_{\Xhbkad}(\overline{\frakU}) = \hbR[\frac1p].\]

The following definition follows from~\cite[Definition 7.1]{Sch13}, which will be used to describe the generalized representations. 
\begin{defi}
An \emph{$\hat{\mO}_{\Xhbkad}^+$-local system} is a sheaf of $\hat{\mO}_{\Xhbkad}^+$-module that is locally on $\Xkadproet$ free of finite rank.
\end{defi} 

Let 
\[\frakU=\Spec(\frakR)\]
be a small affine open subset of $\frakX$. Since 
\[\hat{\mO}_{\Xhbkad}^+(\overline{\frakU}) =\hbR,\]
one has 
\[\Aut(\overline{\frakU}/\frakU_\hbk)=\Aut(\hbR/\frakR_\hbV).\] 
This automorphic group of the basis space naturally induces a semilinear action of $\Aut(\hbR/\frakR_\hbV)$ on the $\hbR$-module $\Gamma(\overline{\frakU},\mM)$ for any $\hat\mO^+_{\frakU_\hbk^{\mathrm{ad}}}$-local system $\mM$. Then $\Gamma(\overline{\frakU},-)$ forms an equivalent functor mapping from the category of $\widehat{\mO}^+_{\frakU_{\hbk}^{\mathrm{ad}}}$-local systems to that of generalized representations of the fundamental group of $\frakU_k$. 
Let $\{\frakU_i\}_{i\in I}$ be a small affine covering of $\frakX$. 
Then the family of functors $\Gamma(\overline{\frakU_{i}},-)$ forms a functor mapping from the category of $\hat{\mO}_{\Xhbkad}^+$-local systems to that of generalized representations of the fundamental group of $\frakX_k$.\\[0mm]

Conversely, let 
\[\bV=((\bV_i,\rho_i),\beta_{ij})\]
be a generalized representation of the fundamental group of $\frakX_k$, where $(\bV_i,\rho_i)$ is an $\hbR_i$-module with a semilinear action of $\Delta_i$ and $\beta_{ij}$'s are the compatible gluing maps. Then by gluing the $\hat\mO_{\Xhbkad}^+(\frakU_i)$-local system associated to $(\bV_i,\rho_i)$, one gets an $\hat{\mO}_{\Xhbkad}^+$-local system. In conclusion, one gets the following result.
\begin{lem}\label{another description}
There is a natural equivalence mapping from the category of $\hat{\mO}_{\Xhbkad}^+$-local systems to that of the generalized representations of the fundamental group of $\frakX_k$.\\[2mm]
In the rest of this note, we will always identify these two category via this equivalent functor.
\end{lem}

Let $(E,\theta)$ be a small small Higgs bundle and let 
\[\bV=\bD^{Fal}(E,\theta)\]
be the $\hat{\mO}_{\Xhbkad}^+$-local system under $p$-adic Simpson correspondence. Let $\{\frakU_i\}_i$ be a chosen small affine open covering of $\frakX$. Then one has following commutative diagram
\begin{equation}
\xymatrix{
	\bV(\overline{\frakU_{ij}}) \ar@{=}[d] \ar[r]^-{\cong} & E(\frakU_{i})\otimes_{\frakR_i}\hbR_{ij} \ar[d]^-{\beta_{\tau_i,\tau_j}}\\
	\bV(\overline{\frakU_{ij}}) \ar[r]^-{\cong} & E(\frakU_{j})\otimes_{\frakR_j}\hbR_{ij} \\
}
\end{equation}
where 
\[\begin{split}
\frakU_{ij} & =\frakU_i\cap\frakU_j,\\
\overline{\frakU_{ij}} & =\varprojlim\limits_{V/\frakU_{ij,\hbk}^{\mathrm{ad}}:\text{finite etale}} V\in \Xhbkadproet,\\
\hbR_{ij} & =\Gamma(\overline{\frakU_{ij}},\hat{\mO}_{\Xhbkad}^+)	
\end{split}\] 
and $\beta_{\tau_i,\tau_j}$ defined in \eqref{equ:taylorHiggs}.

\subsection{Functors constructed via p-adic period sheaves related to $p$-adic Simpson correspondence}

Here we will recall Scholze's Riemann-Hilbert correspondence, and some other intermediate functors to compare the modified functor in Theorem~\ref{thm:modifiedFunctor} with Scholze's Riemann-Hilbert correspondence.\\

Keep the notation as in Section~\ref{section_discrip_via_Sheaves}. 
Recall form \cite[Definition~6.8]{Sch13} the de Rham period sheaf $\OBdR_{,\Xkad}$ on $\Xkadproet$. It is a sheaf of $\mO_{\Xkad}$-algebras which admits a decreasing filtration $\Fil^\bullet$ and an integrable connection 
\[\nabla\colon \OBdR_{,\Xkad}\rightarrow \OBdR_{,\Xkad}\otimes \Omega_{\frakX_k}^1.\]
Taking the $0$-th grading piece, one gets a period sheaf 
\[\OC_{\Xkad} := \Gr^0\OBdR_{,\Xkad}.\]
The associated graded connection
\begin{equation}\label{equ:gradHiggs}
	\Gr(\nabla)\colon \OC_{\Xkad}\rightarrow \OC_{\Xkad}\otimes\Omega_\frakX^1(-1).	
\end{equation}
forms a Higgs field over $\OC_{\Xkad}$ with value in the Tate twisting $\Omega_\frakX^1(-1)$.

\para{Scholze's Riemann-Hilbert correspondence.} In~\cite[Theorem 7.6]{Sch13}, Scholze constructed a pair of functors via the period ring $\OBdR_{,\Xkad}$, which will be denoted by $\mM^{Sch}$ and $\RH^{Sch}$. This two functors $\mM^{Sch}, \RH^{Sch}$ are between the category of filtered $\mO_{\Xkadet}$-modules with integrable connections and the category of $\bBdRp_{,\Xkad}$-local systems.
\begin{equation*}
\xymatrix@C=1cm{
\left\{\txt{filtered $\mO_{\Xkadet}$-modules with \\ integrable connections}\right\} \ar@/^2pt/[r]^-{\mM^{Sch}} & \left\{\txt{$\bBdRp_{,\Xkad}$-local systems}\right\} \ar@/^2pt/[l]^-{\RH^{Sch}} \\ 
}
\end{equation*}
For any filtered $\mO_{\Xkadet}$-module $\mE$ with integrable connection $\nabla$, denote by
\[ \mM^{Sch}(\mE) := \Fil^0(\nu^*\mE\otimes \OBdR_{,\Xkad})^{\nabla=0},\]
and for any $\bBdRp_{,\Xkad}$-local system $\bM$ denote by
\[\RH^{Sch}(\bM):=\nu_*(\bM\otimes \OBdR_{,\Xkad}).\]

\begin{thm}[{\cite[Theorem 7.6]{Sch13}}] The functor $\mM^{Sch}$ is fully faithful and $\RH^{Sch}$ is a left quasi-inverse of $\mM^{Sch}$.	
\end{thm}

\begin{rmk}
In general the functor $\RH^{Sch}$ is not a right quasi-inverse of the functor $\mM^{Sch}$. But there is a natural embedding of $\bBdRp_{,\Xkad}$-local systems
\[\mM^{Sch}\circ\RH^{Sch}(\bM)\hookrightarrow \bM\]
for any $\bBdRp_{,\Xkad}$-local system $\bM$, which is induced by a natural injection\add{why?} 
\[RH^{Sch}(\bM) \otimes \OBdR_{,\Xkad} \hookrightarrow \bM\otimes\OBdR_{,\Xkad}.\]	
\end{rmk}

\para{Liu and Zhu's Riemann-Hilbert correspondence.} Recall the diagram~\eqref{diag_cano_map}, according to the construction~\cite[Defintion~6.8]{Sch13}, one has 
\[\OBdR_{,\Xhbkad} = \OBdR_{,\Xkad} \mid_{\Xhbkadproet}.\] 
Recall from~\cite[Lemma~3.7]{LZ}, one has 
\[\nu'_*(\OBdR_{,\Xhbkad}) = \mO_{\Xhbkadet}\hat{\otimes}\BdR.\]
We note that $\mO_{\Xhbkadet}\hat{\otimes}\BdR$ is the structure sheaf of the ringed spaces 
\[\mX=(\Xhbkad, \mO_{\Xhbkadet}\hat\otimes \BdR)\]
in Liu and Zhu's paper, for more details see~\cite[Section 3.1]{LZ}. We recall a geometric version of the $p$-adic Riemann-Hilbert correspondence in~\cite[Theorem~3.8]{LZ}. There is a tensor functor $\RH$ mapping from the category of \'etale $\bQ_p$-local systems to the category of vector bundles on the ringed space $\mX$ equipped with a semilinear action of the absolute Galois group and with a filtration and integrable connection satisfying Griffiths' transversality. The following two functors $\mM^{LZ}$ and $\RH^{LZ}$ follow Liu and Zhu's original functor $\RH$. They are between the category of filtered $\mO_{\Xhbkadet}\hat{\otimes}\BdR$-modules with integrable connections and the category of $\bBdRp_{,\Xhbkad}$-local systems
	\begin{equation*}
	\xymatrix@C=1cm{
		\left\{\txt{filtered $\mO_{\Xhbkadet}\hat{\otimes}\BdR$-modules \\ with integrable connection}\right\} \ar@/^2pt/[r]^-{\mM^{LZ}} & \left\{\txt{$\bBdRp_{,\Xhbkad}$-local system}\right\} \ar@/^2pt/[l]^-{\RH^{LZ}} \\ 
	}
	\end{equation*} 
For any filtered $\mO_{\Xhbkadet}\hat{\otimes}\BdR$-module $\mE$ with integrable connection $\nabla$, denote
	\[ \mM^{LZ}(\mE) := \Fil^0(\nu'^*\mE\otimes \OBdR_{,\Xhbkad})^{\nabla=0},\]
	and for any $\bBdRp_{,\Xkad}$-local system $\bM$ denote 
	\[\RH^{LZ}(\bM):=\nu'_*(\bM\otimes \OBdR_{,\Xhbkad}).\]

\begin{rmk} The restriction of $\RH^{LZ}$ on $\bQ_p$-local systems is Liu and Zhu's geometric version of the $p$-adic Riemann-Hilbert correspondence $\RH$. In particular,~\cite[Theorem 3.8(iii)]{LZ} implies that $\RH^{LZ}$ preserves the ranks for $\bQ_p$-local systems on $\Xkadet$.
\end{rmk}

\begin{remark}
The pairs $(\mM^{Sch},\RH^{Sch})$ and $(\mM^{LZ},\RH^{LZ})$ are related as follows. For a filtered $\mO_{\Xkadet}$-module $\mE$ with integrable connection, by restricting on $\Xhbkadet$ and tensoring $\mO_{\Xkadet}\hat{\otimes}\BdR$, one gets a filtered $\mO_{\Xkadet}\hat{\otimes}\BdR$-module $\mE\mid_{\Xhbkadet} \hat\otimes \BdR$ with integrable connection. Since the functor $(\cdot)^{\nabla=0}$ of taking the flat sections and the functor $\Fil^0(\cdot)$ of taking $0$-th filtration commute with the restriction functor, one has the following equality between $\bBdRp_{,\Xhbkad}$-local systems
\[\mM^{LZ}\left(\mE\mid_{\Xhbkadet} \hat\otimes \BdR \right) = \mM^{Sch}(\mE)\mid_{\Xhbkadproet}.\]
Conversely, for any $\bBdRp_{,\Xkad}$-local system $\bM$, there is an injective map between filtered $\mO_{\Xkadet}\hat{\otimes}\BdR$-modules 
\[\RH^{Sch}(\bM)\mid_{\Xhbkadet} \hat\otimes \BdR \hookrightarrow \RH^{LZ}(\bM\mid_{\Xhbkadproet}).\]
\end{remark}
\ \\
Taking the $0$-th grading pieces, one gets 
\[\OC_{\Xhbkad} = \OC_{\Xkad} \mid_{\Xhbkadproet}.\]
Similarly, by using this sheaf, one constructs two functors $\bD^{\OC}, \mH^{\OC}$ between the category of Higgs bundle $(E,\theta)$ over $\Xhbkadet$ and the category of generalized $\bQ_p$-representations(i.e. the $\hat{\mO}_{\Xhbkad}$-local systems).
\begin{equation*}
\xymatrix@C=2cm{
\left\{\txt{Higgs bundles \\ over $\Xhbkadet$}\right\} \ar@/^2pt/[r]^-{\bD^{\OC}} & \left\{\txt{generalized $\bQ_p$-representations}\right\} \ar@/^2pt/[l]^-{\mH^{\OC}}\\ 
}
\end{equation*} 

For any Higgs bundle $(E,\theta)$ over $\Xhbkadet$, denote
\begin{equation} \label{deffunctor3-1}
	\bD^{\OC}(E,\theta) := (\nu'^*E\otimes \OC_{\Xhbkad})^{\theta=0},
\end{equation}
and conversely for any $\hat{\mO}_{\Xhbkad}$-local system $\bV$ denote 
\begin{equation} \label{deffunctor3-2}
	\mH^{\OC}(\bV):=\nu'_*(\bV\otimes \OC_{\Xhbkad}).
\end{equation}

\begin{remark}
	If $\bV$ is the generalized $\bQ_p$-representation attached to a $\bQ_p$-local system $\bL$ on $\frakX_k$ then 
	\[\mH^{\OC}(\bV)=(\mH(\bL),\vartheta_\bL),\]
	where $(\mH(\bL),\vartheta_\bL)$ is the nilpotent Higgs bundle defined in~\cite[Theorem 2.1]{LZ}. This follows from the definition of $\mH$.
\end{remark}

\begin{remark} 
The pairs $(\mM^{LZ},\RH^{LZ})$ and $(\bD^{\OC},\mH^{\OC})$ are related as follows. Let $(\mV,\nabla,\Fil)$ be a filtered $\mO_{\Xkadet}\hat{\otimes}\BdR$-module $\mE$ with integrable Griffiths connection. The $0$-th grading piece $\Gr^0(\mV,\nabla,\Fil)$ forms a Higgs bundle over $\Xhbkadet$. Since 
\[\OC_{\Xhbkad}=\Gr^0(\OBdR_{,\Xhbkad}),\]
there is a nature embedding maps $\Fil^0\left((\mV,\nabla,\Fil) \otimes \OBdR\right) \rightarrow \Gr^0(\mV,\nabla,\Fil)\otimes \OC$ which induces a natural injective map
\[ \Gr^0 \circ \mM^{LZ} (\mV,\nabla,\Fil) \hookrightarrow \bD^{\OC} \circ \Gr^0 (\mV,\nabla,\Fil).\]
This injection is actually an isomorphism if $(\mV,\nabla,\Fil)$ comes from filtered $\mO_{\Xhbkadet}$-module with integrable Griffiths connection. This follows~\cite[Theorem 7.6]{Sch13}.
For any $\bBdRp_{,\Xhbkad}$-local system $\bM$, the $0$-th grading piece is just a generalized $\bQ_p$-representation. The natural map 
\[\Gr^0(\bM\otimes\OBdR) \rightarrow \Gr^0(\bM) \otimes \OC\]
induces
\[ \Gr^0 \circ \RH^{LZ} (\bM)\hookrightarrow \mH^\OC \circ \Gr^0 (\bM).\]
\end{remark}

\ \\

We recall from~\cite[corrollary 6.15]{Sch13} the local structure of $\OC_{\Xhbkad}$. For any open subset
\[\frakU =\Spec(\frakR)\subset \frakX,\]
recall the object 
\[\overline{U} = \varprojlim\limits_{V/\frakU_{\hbk}:\text{finite etale}} V \in \Xhbkadproet\]
defined in \eqref{universal-covering}. One has 
\[\hat{\mO}_{\Xhbkad}^+(\overline{U})=\hbR \text{ and } \hat{\mO}_{\Xhbkad}(\overline{U})=\hbR\otimes \bQ_p.\] 
Let 
\[\TT=\{T_1,\cdots,T_d\}\]
be some chosen local parameters of $\frakR$ and 
\[\tau\colon \frakR \rightarrow A_2(\frakR)[\pi]\]
be a chosen lift in \eqref{equ:lift_tau}. Denote
\begin{equation}\label{equ:u_i}
u_i =- T_i\otimes 1 + 1\otimes \tau(T_i)\in \OBinfRh.
\end{equation}
Similarly as in~\cite[corrollary 6.15]{Sch13}, one has an isomorphism of $\hbR$-algebras
\begin{equation}\label{equ:localseries}
\OC_{\Xhbkad}(\overline{U})=\OCRh\cong \hbRK[\frac{u_1}{\xi},\cdots,\frac{u_d}{\xi}].
\end{equation}
\ \\

In the following, we introduce another larger period sheaf of $\mO_{\Xkad}$-algebras $\tOC_{\Xhbkad}$ over $\Xhbkadproet$. Denote
\[\tOCRh:= \bigcup_{\alpha>\alpha_0} \hbR[[p^{\alpha}\frac{u_1}{\xi},\cdots,p^{\alpha}\frac{u_d}{\xi}]][\frac1p] \subset \hbRK[[\frac{u_1}{\xi},\cdots,\frac{u_d}{\xi}]].\]
where $\alpha_0$ is the nonnegative real number given in Definition~\ref{def:verysmall}. one easily checks that the ring $\tOCRh$ satisfies the following facts. 
\begin{lem}\label{lem2.21} Keep notation as above.
\begin{itemize}
\item[i).] Up to a canonical isomorphism $\tOCRh$ does not depend on the choice of $\tau$, and $\OCRh$ is a subring of $\tOCRh$;
\item[ii).] The action of $\Delta$ and the Higgs field $\Gr(\nabla)$ naturally extend to $\tOCRh$;
\item there exists an $\hbRK$-linear projection 
\[\pr \colon \tOCRh \twoheadrightarrow \hbRK\]
sending $\frac{u_i}{\xi}$'s to zero. 
\end{itemize}
\end{lem}

By gluing the local ring $\tOCRh$, one gets a period sheaf of $\mO_{\Xhbkad}$-algebras over $\Xhbkadproet$, we denote it by $\tOC_{\Xhbkad}$. Follows Lemma~\ref{lem2.21}, one gets following basic facts of $\tOC_{\Xhbkad}$.
\begin{lem} Keep notation as above.
\begin{itemize}
\item[i).] $\OC_{\Xhbkad}\subset \tOC_{\Xhbkad}$ is a subsheaf of algebras;
\item[ii).] The action of $\Delta$ and the Higgs field $\Gr(\nabla)$ naturally extend to $\tOC_{\Xhbkad}$.
\end{itemize}
\end{lem}

By using the sheaf $\tOC_{\Xhbkad}$, similar constructions as in~\eqref{deffunctor3-1} and \eqref{deffunctor3-2}, one constructs following two functors between the category of Higgs bundle $(E,\theta)$ over $\Xhbkadet$ and the category of generalized $\bQ_p$-representations(i.e. the $\hat{\mO}_{\Xhbkad}$-local systems).
	\begin{equation*}
	\xymatrix@C=2cm{
		\left\{\txt{Higgs bundles \\ over $\Xhbkadet$}\right\} \ar@/^2pt/[r]^-{\bD^{\tOC}} & \left\{\txt{generalized $\bQ_p$-representations}\right\} \ar@/^2pt/[l]^-{\mH^{\tOC}} \\ 
	}
	\end{equation*} 

\begin{remark} The nature embedding $\OC_{\Xhbkad}\subset \tOC_{\Xhbkad}$ induces $\bD^{\OC}(E,\theta) \subseteq \bD^{\tOC}(E,\theta)$ for any Higgs bundle $(E,\theta)$ over $\Xhbkadet$ and $\mH^{\OC}(\bV) \subseteq \mH^{\tOC}(\bV)$ for any generalized $\bQ_p$-representations $\bV$.
\end{remark}

Let us sum up notation in this section in the following diagram. 
\begin{equation*}
\xymatrix{
\left\{\txt{filtered $\mO_{\Xkadet}$-module with \\ integrable connection}\right\} \ar@/^2pt/[r]^-{\mM^{Sch}} \ar[d]^{-\hat{\otimes}\BdR} & \left\{\txt{$\bBdRp_{,\Xkad}$-local system}\right\} \ar@/^2pt/[l]^-{\RH^{Sch}} \ar[d]^{\text{restriction}}
\\ 
\left\{\txt{filtered $\mO_{\Xhbkadet}\hat{\otimes}\BdR$-modules \\ with integrable connection}\right\} \ar@/^2pt/[r]^-{\mM^{LZ}} \ar[d]^{\Gr^0} & \left\{\txt{$\bBdRp_{,\Xhbkad}$-local system}\right\} \ar@/^2pt/[l]^-{\RH^{LZ}} \ar[d]^{\Gr^0} 
\\ 	
 \left\{\txt{Higgs bundles \\ over $\Xhbkadet$}\right\} \ar@/^2pt/[r]^-{\bD^{\OC},\bD^{\tOC}} & \left\{\txt{generalized $\bQ_p$-representations}\right\} \ar@/^2pt/[l]^-{\mH^{\OC},\mH^{\tOC}} 
\\ 
}
\end{equation*}
Let us recall from Definition~\ref{def:FalFun} that the two functors $\bD^{Fal}$ and $\mH^{Fal}$ are defined between the following categories
\begin{equation*}
		\xymatrix@C=1cm{
			\left\{\txt{very small Higgs \\ bundles over $\Xhbkadet$}\right\} \ar@/^2pt/[r]^-{\bD^{Fal}} & \left\{\txt{very small generalized\\ $\bQ_p$-representations}\right\} \ar@/^2pt/[l]^-{\mH^{Fal}} \\ 
		}
\end{equation*}

\section{Comparisons}

Here we study some relations between functors $\bD^{Fal}$, $\mM^{Sch}$, $\bD^{\OC}$, $\bD^{\tOC}$, $\mH^{Fal}$, $\RH^{Sch}$, $\mH^{\OC}$ and $\mH^{\tOC}$ given in section~\ref{section_preliminaries}.

For very small generalized representations, there are natural transformations
\[\mH^{\OC} \rightarrow \mH^{\tOC} \rightarrow \mH^{Fal}.\]
We summon up some the relations between $\mH^{\OC}$, $\mH^{\tOC}$ and $\mH^{Fal}$ in the following result.
\begin{prop}\label{mainthm1} Let $\bV$ be a generalized $\bQ_p$-representation. Then
\begin{itemize}
	\item[i).] the inclusion $\OC_{\Xkad}\hookrightarrow\tOC_{\Xhbkad}$ induces a natural injective map 
	\[\mH^{\OC}(\bV) \subset\mH^{\tOC}(\bV)\]
	and $\mH^{\OC}(\bV)$ is the maximal nilpotent sub-Higgs bundle of $\mH^{\tOC}(\bV)$;
	\item[ii).] Suppose moreover that $\bV$ is very small. Then there is a natural isomorphism
	\[\mH^{\tOC}(\bV) \rightarrow \mH^{Fal}(\bV).\] 
\end{itemize} 
\end{prop}

\begin{cor}\label{Cor:LiuZhu} \footnote{There is an overlap between Corollary 3.2 and Section-7 in a recent preprint by Yupeng Wang.
Our note was submitted in arXiv (arXiv:2012.02058v1) on December 3, 2020. Slightly late, Yupeng Wang submitted his paper entitled ``A $p$-adic Simpson correspondence for rigid analytic varieties" in arXiv (arXiv:2101.00263v1) on January 1, 2021. There he also proved the result in Corollary~\ref{Cor:LiuZhu}.} 
For a $\bQ_p$-local system $\bL$ on $\frakX_k$, Liu and Zhu's $p$-adic Simpson correspondence~\cite[Theorem 2.1]{LZ} coincide with modified Faltings' $p$-adic Simpson correspondence.
\end{cor}
\begin{proof} Since both 
	\[\mH^{\OC}(\bV)=(\mH(\bL),\vartheta_\bL)\]
	 and $\mH^{Fal}(\bV)$ have the same rank with $\bL$, Proposition~\ref{mainthm1} implies they coincide with each other.
\end{proof}

\begin{proof}[Proof of Proposition~\ref{mainthm1}] The inclusion $\mH^{\OC}(\bV) \subset\mH^{\tOC}(\bV)$ follows from the definitions. Firstly, we show the nilpotency of $\mH^{\OC}(\bV)$. Since a Higgs bundle being nilpotent is a local property, it is sufficient to show the local case. We may assume that $\bV$ is a finite $\hbRK$-module with semilinear action of 
	\[\Delta=\Gal(\hbR/\hRbarV),\]
where 
 \[\frakU=\Spec(\frakR)\]
is a small affine open subset of $\frakX$. Recall~\eqref{equ:localseries} and \eqref{equ:gradHiggs}, one has 
\[\OCRh\cong \hbRK[\frac{u_1}{\xi},\cdots,\frac{u_d}{\xi}]\]
with $\hbRK$-linear Higgs field 
	\[\Gr(\nabla)\colon \OCRh\rightarrow \OCRh\otimes \Omega^1_\frakU(-1)\]
on $\OCRh$ given by 
	\[\Gr(\nabla)\left( \left(\frac{u_1}{\xi}\right)^{\alpha_1} \cdots \left(\frac{u_d}{\xi}\right)^{\alpha_d}\right) = \sum_i -\alpha_i\left( \left(\frac{u_1}{\xi}\right)^{\alpha_1} \cdots \left(\frac{u_i}{\xi}\right)^{\alpha_i-1} \cdots \left(\frac{u_d}{\xi}\right)^{\alpha_d}\right) \otimes \mathrm{d}T_i\cdot \xi^{-1}.\]
In particular,
\begin{equation}\label{partialHiggs}
	\Gr(\nabla)(\xi\partial_i) \left( \left(\frac{u_1}{\xi}\right)^{\alpha_1} \cdots \left(\frac{u_d}{\xi}\right)^{\alpha_d}\right) = -\alpha_i \left(\frac{u_1}{\xi}\right)^{\alpha_1} \cdots \left(\frac{u_i}{\xi}\right)^{\alpha_i-1} \cdots \left(\frac{u_d}{\xi}\right)^{\alpha_d}.
\end{equation} 
On one hand, the underlying module of $\mH^{\OC}(\bV)$ is defined to be the $\Delta$-invariant part of $\bV\otimes\OCRh$, for any element $e\in\mH^{\OC}(\bV)$ it can be written in the form
	\[e = \sum_{|I|<N} v_I\otimes \left(\frac{u}{\xi}\right)^{I},\]
where $I=(i_1,\cdots,i_d)\in\bN^d$ with 
\[|I|:=i_1+\cdots+i_d,\]
$v_I$ is some element in $\bV$, $N$ is some positive integer, and 
\[\left(\frac{u}{\xi}\right)^{I}:=\prod_\ell \left(\frac{u_\ell}{\xi}\right)^{i_\ell}.\]
On the other hand, the Higgs field $\theta^{\OC}$ on $\mH^{\OC}$ is induced by the grading $\Gr(\nabla)$ of the connection $\nabla$, one has 
\[\theta^{\OC}(e) := \sum_{|I|<N} v_I\otimes \Gr(\nabla) \left(\frac{u}{\xi}\right)^{I}.\]
Thus \eqref{partialHiggs} implies 
	\[\theta^{\OC}((\xi\partial)^J)(e)=0, \quad \text{for any $J$ with } |J|\geq N.\]
Hence $\theta^{\OC}$ is nilpotent.\\[0mm]
	
Secondly, we show that $\mH^{\OC}(\bV)$ is the maximal nilpotent part of $\mH^{\tOC}(\bV)$. We only need to show that any element $\eta\in \mH^{\tOC}(\bV)$ which is nilpotent under the Higgs field $\theta^{\tOC}$ is contained in $\mH^{\OC}(\bV)$. We write it as form 
	 \[\eta = \sum_{I\in\bN^d} w_I \otimes \left(\frac{u}{\xi}\right)^{I}.\]
Since $\eta$ is nilpotent element, there exists some $N$ such that for any $J\in \bN^d$ with $|J|\geq N$, one has $\theta^{\tOC}((\xi\partial)^J)(\eta)=0$, where
\[\theta^{\tOC}((\xi\partial)^J)(\eta) := \sum_{I\in\bN^d} w_I \otimes \Gr(\nabla)((\xi\partial)^J) \left(\frac{u}{\xi}\right)^{I}.\]
By~\eqref{partialHiggs}, for any $I$ with $I\geq N$, 
\[w_I=0.\]
There are only finitely many nonzero terms in the series, which means that $\eta$ is also contained in $\mH^{\OC}(\bV)$.\\[0mm] 
	
Finally, We construct the natural isomorphism 
\[\mH^{\tOC}(\bV) \rightarrow \mH^{Fal}(\bV),\]
for a very small generalized $\bQ_p$-representation $\bV$. Let us denote $\theta$ the Higgs field in the Higgs bundle $\mH^{Fal}(\bV)$. Because $\bV$ is very small, according the explicit construction~\eqref{underlyingMod} and \eqref{equ:galoisAction} of the equivalent functor in Theorem~\ref{thm:modifiedFunctor}, over a small affine open subset 
\[\frakU=\Spec(\frakR),\]
one has 
\[\bV(\overline{\frakU})=\mH^{Fal}(\bV)(\frakU)\otimes \hbRK.\]
Extending the action $\rho$ of $\Delta$ on $\bV(\overline{\frakU})$ semilinearly onto 
\[\bV(\overline{\frakU})\otimes\tOCRh = \mH^{Fal}(\bV)(\frakU)\otimes \tOCRh.\]
That is for any 
\[e\otimes r\in\mH^{Fal}(\bV)(\frakU)\otimes\tOCRh\]
one has 
\[ \rho(\sigma)(e\otimes r) = \exp\left[\sum_{i=1}^{d} \theta(\xi\partial_i)\otimes\left(\left(\frac{u_i}{\xi}\right)^{\sigma} - \left(\frac{u_i}{\xi}\right)\right)\right](e\otimes \sigma(r)).\]
In particular, one gets a $\Delta$-invariant element
	\[\hat{e}:= \exp\left[-\sum_{i=1}^{d} \theta(\xi\partial_i)\otimes\frac{u_i}{\xi}\right](e\otimes1)\in \mH^{Fal}(\bV)(\frakU)\otimes \widetilde{\tOCRh}.\]
This is because 
\[\rho(\sigma)(\hat{e}) = \exp\left[-\sum_{i=1}^{d} \theta(\xi\partial_i)\otimes\left(\frac{u_i}{\xi}\right)^\sigma\right] \Big( \rho(\sigma)(e\otimes1)\Big)=\hat{e}\]
for any $\sigma\in\Delta$.
In other words, one has $\hat{e}\in \mH^{\tOC}(\bV)$. The map $\hat{(\cdot)}$ is clearly injective and $\hRbark$-linear mapping from $\mH^{Fal}(\bV)(\frakU)$ to $\mH^{\tOC}(\bV)(\frakU)$. We only need to show it is also surjective. that is that every element $\eta\in \mH^{\tOC}(\bV)$ is of the form
	\[\hat{e}:= \exp\left[-\sum_{i=1}^{d} \theta(\xi\partial_i)\otimes\frac{u_i}{\xi}\right](e\otimes1),\]
for some $e\in \mH^{Fal}(\bV)$. Let $\{e_i\}$ be a basis of $\mH^{Fal}(\bV)$. Then $\{\hat{e_i}\}$ is also a basis of $\mH^{Fal}(\bV)\otimes\tOCRh$. We may write 
\[\eta=\sum_i \hat{e_i} \cdot \eta_i\]
with $\eta_i\in \tOCRh$. Since $\eta$ is $\Delta$ invariant, one obtains that $\eta_i^\sigma=\eta_i$ for all $\sigma\in \Delta$ and all $i$. Thus 
\[\eta_i\in \tOCRh^\Delta = \hRbark.\]
 So the element $e:=\sum_i e_i\cdot\eta_i$ is contained in $\mH^{Fal}(\bV)$ satisfies 
 \[\widehat{e}=\sum_i \hat{e_i}\cdot\eta_i=\eta.\qedhere\]
\end{proof}

Next, we show that the two functors $\bD^{\tOC}$ and $\bD^{Fal}$ coincide with each other for the small Higgs bundles. 
\begin{thm}\label{thm:main3}
	Let $(E,\theta)$ be a very small Higgs bundle. Then we have an isomorphism between generalized $\bQ_p$-representations.
	\[\bD^{\tOC}(E,\theta) \cong \bD^{Fal}(E,\theta).\]
\end{thm}

Before giving the proof, we need to introduce a projection map $\mathrm{pr}_\tau$. Let $\frakR$ be a small smooth $W$-algebra of relative dimension $d$. Choose some local parameters 
\[\TT=\{T_1,\cdots,T_d\}\]
of $\frakR$ and a lift 
\[\tau\colon \frakR \rightarrow A_2(\frakR)[\pi]\]
as in \eqref{equ:lift_tau}. Recall the definition 
\[\tOCRh:= \bigcup_{\alpha>\alpha_0} \hbR[[p^{\alpha}\frac{u_1}{\xi},\cdots,p^{\alpha}\frac{u_d}{\xi}]][\frac1p] \subset \hbRK[[\frac{u_1}{\xi},\cdots,\frac{u_d}{\xi}]].\]
There exists an $\hbRK$-linear projection 
\[\tOCRh \twoheadrightarrow \hbRK\]
sending $\frac{u_i}{\xi}$'s to zero. 
For any Higgs bundle $(E,\theta)$ over $\hRbark$ with the Higgs field $\theta$ contained in $\End(E)\otimes\Omega^1_{\frakR/V}\cdot\xi^{-1}$, there exists an $\hbRK$-linear morphism
\begin{equation} \label{proj: Higgs}
\pr_{\tau} \colon E\otimes \tOCRh \rightarrow E\otimes \hbRK.
\end{equation}
We note that this morphism depends on $\tau$ but not on $\TT$. We show that its restriction on the submodule which killed by $\theta$ is an isomorphism. 
\begin{lem}\label{lem: main-02} Fix a lift $\tau\colon \frakR \rightarrow A_2(\frakR)[\pi]$. Let $(E,\theta)$ be a very small Higgs bundle over $\hRbark$. Then
\begin{itemize}
\item[i).] the submodule $\left(E\otimes \tOCRh\right)^{\theta=0}$ has an $\hbRK$-module structure and it contains elements of form
	\[\widetilde{e}=\sum_{I} \frac{\theta^I(\xi\partial)(e)}{I!} \otimes \left(\frac{u}{\xi}\right)^I.\]
\item[ii).] The restriction of the map in \eqref{proj: Higgs} on the submodule in $\mathrm{(i)}$
	\[\pr_{\tau} \colon \left(E\otimes_{\frakR} \tOCRh\right)^{\theta=0} \longrightarrow E\otimes_{\frakR} \hbRK\]
	is an $\hbRK$-linear isomorphism.
\item[iii).] If $\tau'$ is another lift, then there is a commutative diagram:
	\begin{equation}
		\xymatrix{
			\left(E\otimes_{\frakR} \tOCRh\right)^{\theta=0} \ar[rr]^-{\pr_{\tau}}_-{\cong} \ar@{=}[d]
			&&
			E\otimes_{\frakR}\hbRK \ar[d]^-{\beta_{\tau,\tau'}}_-{\cong}\\
			\left(E\otimes_{\frakR} \tOCRh\right)^{\theta=0} \ar[rr]^-{\pr_{\tau'}}_-{\cong} &&
			E\otimes_{\frakR} \hbRK\\
		}
	\end{equation}
\end{itemize}
\end{lem}

\begin{proof} (i) Since 
	\[\tOCRh^{\theta=0} = \hbRK,\]
	one obtains that $\left(E\otimes_{\frakR} \tOCRh\right)^{\theta=0}$ is an $\hbRK$-module. Note that 
	\[\theta\Big(\left(\frac{u}{\xi}\right)^I\Big) = -\sum_{\ell=1}^{d} i_\ell \left(\frac{u}{\xi}\right)^{I-I_\ell}\cdot \frac{\mathrm{d}T_\ell}{\xi}\]
	and 
	\[\theta (\theta^I(\xi\partial)(e)) = \sum_{\ell=1}^{d}\theta^{I+I_\ell}(\xi\partial)(e) \cdot \frac{\mathrm{d} T_\ell}{\xi},\]
	where $I=(i_1,\cdots,i_d)$ and $I_\ell=(0,\cdots,\underset{\ell\text{-th}}{1},\cdots,0)\in \bN^d$.
	Thus 
	\[\theta(\widetilde{e})=0.\]

(ii) Firstly, the map $\pr_\tau$ is surjective. For any $e \in V$, 
\[\theta\left(\tilde{e}\right)=0 \quad \text{ and } \quad \pr_\tau\left(\tilde{e}\right) = e\otimes 1,\]
Since $\pr$ is $\hbR$-linear and $E(\frakU)\otimes \hbR$ is generated by the elements of form $e\otimes1$, thus $\pr$ is surjective.

Secondly, the map $\pr_i$ is injective. By~\eqref{equ:localseries}, any element in $E\otimes \tOCRh$ can be written uniquely as form $\sum\limits_{I} \frac{e_{I}}{I!} \cdot \left(\frac{u}{\xi}\right)^I$ with $e_I\in V \otimes \RK$. Suppose that
\[\theta\left(\sum_{I} \frac{e_{I}}{I!} \cdot \left(\frac{u}{\xi}\right)^I\right)=0\quad \text{ and } \quad \pr\left(\sum_{I} \frac{e_{I}}{I!} \cdot \left(\frac{u}{\xi}\right)^I\right) = 0.\]
On one hand, the first equality implies that
\[e_{I+I_k} = \theta(\xi\partial_k)(e_I),\]
for all $I\in \bN^d$ and $k\in\{1,2\cdots,d\}$.
on the other hand, the second equality implies that $e_0$ equals zero. Thus 
\[e_I=0\]
for all $I$. So $\pr$ is injective.

(iii) Since 
\[\frac{u}{\xi}:=\frac{-T\otimes1+\tau(T)}{\xi} =\frac{-T\otimes1+\tau'(T)}{\xi} + \frac{\tau(T)-\tau'(T)}{\xi},\]
 one has 
 \[\pr_{\tau'}(\frac{u}{\xi}) = \frac{\tau(T)-\tau'(T)}{\xi}.\]
 Thus
\begin{equation*}
\begin{split}
\pr_{\tau'}(\widetilde{e}) & = \sum_{I} \frac{\theta^I(\xi\partial)(e)}{I!} \otimes \left(\frac{\tau(T) -\tau'(T)}{\xi}\right)^I \\
& = \beta_{\tau,\tau'}(e\otimes1) = \beta_{\tau,\tau'} \circ \pr_{\TT}(\widetilde{e})
\end{split}
\end{equation*}
Since elements of form $\widetilde{e}$ generates the whole submodule $\left(E\otimes_{\frakR} \tOCRh\right)^{\theta=0}$ as an $\hbRK$-module, the diagram commutes.
\end{proof}

\begin{cor} If the very small Higgs bundle $(E,\theta)$ is nilpotent. Then
	\begin{itemize}
		\item[(i)] the submodule $\left(E\otimes \OCRh\right)^{\theta=0} = \left(E\otimes \tOCRh\right)^{\theta=0}$
		\item[(ii)] The restriction of the map 
		\[\pr_{\tau} \colon \left(E\otimes_{\frakR} \OCRh\right)^{\theta=0} \longrightarrow E\otimes_{\frakR} \hbRK\]
		is an $\hbRK$-linear isomorphism.
	\end{itemize}
\end{cor}

\begin{proof}[Proof of Theorem~\ref{thm:main3}]
We firstly need to find a morphism mapping from 
\[\bD^{\tOC}(E,\theta):=(E\otimes_{\OfrakX} \tOC)^{\theta=0}\]
to $\bD^{Fal}(E,\theta)$. It is sufficient to construct compatible local $\Delta$-isomorphisms
\[\left(E(\frakU)\otimes_{\frakR} \OCRh\right)^{\theta=0} \longrightarrow E(\frakU)\otimes_{\frakR} \hbRK\]
for each small open affine subset 
\[\frakU=\Spec(\frakR)\subset\frakX.\]
This is because $\bD^{Fal}(E,\theta)$ is glued from the local representations $E(\frakU_i)\otimes_{\frakR_i} \hbR_i[\frac1p]$ via Taylor formula $\beta_{\tau_i,\tau_j}$~\eqref{equ:taylorHiggs}.

Fix a local lift $\tau$ on each $\frakU$. By Lemma~\ref{lem: main-02}, there does exist an $\hbRK$-linear isomorphism
\[\pr_{\tau}\colon \left(E(\frakU)\otimes_{\frakR} \OCRh\right)^{\theta=0} \longrightarrow E(\frakU)\otimes_{\frakR} \hbRK\]
for each $\frakU$, and these isomorphisms can be glued into a global isomorphism
\[\pr \colon \left(E\otimes \OC\right)^{\theta=0} \rightarrow \bD^{Fal}(E,\theta).\]

In the following, we only need to show that $\pr_{\tau}$ preserves the $\Delta$-actions on both sides. This is equivalent to checking that the following diagram commutes for all $\sigma\in \Delta$
\begin{equation*}
\xymatrix{
\left(E(\frakU)\otimes_{\RK} \OCRh\right)^{\theta=0} \ar[r]^-{\pr_{\tau}} \ar[d]^-{\sigma}
& E(\frakU) \otimes_{\RK} \hbRK \ar[d]^-{\sigma} \\
\left(E(\frakU)\otimes_{\RK} \OCRh\right)^{\theta=0} \ar[r]^-{\pr_{\tau}}
& E(\frakU)\otimes_{\RK} \hbRK \\
}
\end{equation*}
By lemma~\ref{lem: main-02}, one has that $\left(E(\frakU) \otimes_{\RK} \OCRh\right)^{\theta=0}$ is generated by the elements of form
\[\widetilde{e}=\sum_{I} \frac{\theta^I(\xi\partial)(e)}{I!} \otimes \left(\frac{u}{\xi}\right)^I.\]
We only need to check that,for all $e\in E$
\[\pr_{\tau}\circ\sigma (\widetilde{e}) = \sigma\circ \pr_{\tau}(\widetilde{e}).\]
By the definition of $\pr_{\tau}$, one has 
\[\pr_{\tau}(\widetilde{e})=e\otimes1.\]
By equation~\eqref{equ:galoisAction}
\[\sigma\circ\pr_{\tau}(\widetilde{e}) = \sum_{I} \frac{\theta^I(\xi\partial)(e)}{I!} \left(\frac{\tau(T)^\sigma-\tau(T)}{\xi}\right)^I.\]
On the other hand, since 
\[\pr_\tau\circ\sigma(\frac{u}{\xi})=\frac{\tau(T)^\sigma-\tau(T)}{\xi},\]
one has
\[ \pr_\tau\circ\sigma \left(\sum_{I} \frac{\theta^I(\xi\partial)(e)}{I!} \otimes \left(\frac{u}{\xi}\right)^I \right) = \sum_{I} \frac{\theta^I(\xi\partial)(e)}{I!} \left(\frac{\tau(T)^\sigma-\tau(T)}{\xi}\right)^I. \]
This shows that the morphism preserves the $\Delta$-actions.
\end{proof}

By using Theorem~\ref{thm:main3}, we get our main result.
\begin{thm}\label{thm:main2}
Let $\mE$ be a filtered $\mO_{\Xkadet}$-modules with integrable Griffith connection. Denote \[\mM:=\mM^{Sch}(\mE)\mid_{\Xhbkadproet} \text{ and } (E,\theta):=\Gr^0(\mE\hat{\otimes}\BdR).\]
Then there is an isomorphism of generalized representations
	\[\mM/\xi\mM = \bD^{Fal}(E,\theta).\]
In other words, the following diagram commutes
\begin{equation*}
	\xymatrix{
		\left\{\txt{filtered $\mO_{\Xkadet}$-module with \\ integrable connection}\right\} \ar[d]^{-\hat{\otimes}\BdR} \ar[r]^-{\mM^{Sch}} & \left\{\txt{$\bBdRp_{,\Xkad}$-local system}\right\} \ar[d]^{\text{restriction}}
		\\ 
		\left\{\txt{filtered $\mO_{\Xhbkadet}\hat{\otimes}\BdR$-modules \\ with integrable connection}\right\} \ar[d]^{\Gr^0} & \left\{\txt{$\bBdRp_{,\Xhbkad}$-local system}\right\} \ar[d]^{\Gr^0} 
		\\ 	
		\left\{\txt{Higgs bundles \\ over $\Xhbkadet$}\right\} \ar@/^2pt/[r]^-{\bD^{Fal}} & \left\{\txt{generalized $\bQ_p$-representations}\right\} 
		\\ 
	}
\end{equation*}
\end{thm}

\begin{proof}
The surjective map $\OBdR\rightarrow\OC$ induces a nature injective map 
 \[\mM/\xi\mM \rightarrow \bD^{\OC}(E,\theta).\]
Consider the following maps
\[\mM/\xi\mM \hookrightarrow \bD^{\OC}(E,\theta) \rightarrow \bD^{\tOC}(E,\theta)\rightarrow \bD^{Fal}(E,\theta).\]
By Scholze's theorem the rank of $\mM/\xi\mM$ is the same as that of $\mE$. On the other hand, the rank of $\bD^{\OC}(E,\theta)$ is smaller than or is equal to that of $(E,\theta)$. Thus the map 
 \[\mM/\xi\mM \rightarrow \bD^{\OC}(E,\theta)\]
is actually an isomorphism. Since $(E,\theta)$ is graded, one has 
\[\bD^{\OC}(E,\theta)=\bD^{\tOC}(E,\theta).\]
Thus the Theorem follows from Theorem~\ref{thm:main3}.
\end{proof}

\begin{question} Suppose that $\bM$ is a $\bBdRp$-local system induced from a $\bZ_p$-local system over $\Xkadet$. Consider the sub-$\bBdRp$-local system (Remark~) 
	\[\bM':=\mM^{Sch}\circ\RH^{Sch}(\bM) \subset \bM.\] When is the generalized representation $\bM'/\xi\bM'$ a genuine representation?
\end{question}

\section{de Rham sub-local systems}\label{sec:subrep}

We give an application of Theorem~\ref{thm:main2} on the relation between the Scholze's functor $\RH^{Sch}$ and Faltings' functor $\mH^{Fal}$. This is an easiest case in testing Faltings' conjecture and partially answers the following question: how can we reconstruct a given $\bZ_p$-local system over $\frakX_k$~\addd{\cite[Definition 8.1]{Sch13}} from its corresponding de Rham bundle $\RH^{Sch}(\bL\otimes\bBdRp)$?\\[0mm]

For a $\bZ_p$-local system $\bL$ over $\frakX_k$, there are two ways to construct Higgs bundles with trivial Chern classes. On one hand, we may assume $\bL$ to be very small after taking an \'etale base change on $\frakX_k$ and obtain the Higgs bundle 
\[(E,\theta)_{\bL}:=\mH^{Fal}(\bL\otimes \hat{\mO}_{\Xhbkad}^+).\]
Note that $(E,\theta)_\bL$ has trivial Chern classes and is slope semistable with respect to any ample line bundle, since it corresponds to a genuine representation under Faltings' functor. On the other hand, Scholze defined a filtered $\mO_\Xkadet$-module with integrable connection 
\[(V,\nabla,\Fil)=\RH^{Sch}(\bM)\]
attached to $\bL$, where $\bM=\bL\otimes\bBdRp_{,\Xhbkad}$. Subsequently, we have an associated $\mO_\Xkadet$-Higgs bundle 
\[(E,\theta)_\bL^{\mO\bB_{\rm dR}} := \Gr_{\xi}\circ\RH^{Sch}(\bM).\]
where 
\[\Gr_\xi(V,\nabla,\Fil) := (\Gr_\Fil(V),\Gr(\nabla)\cdot\xi^{-1}).\]
Indeed $(E,\theta)_\bL^{\mO\bB_{\rm dR}}$ has trivial Chern classes as $\RH^{Sch}(\bL)$ is a filtered de Rham bundle defined over a field of characteristic zero and the eigenvalues of the residue of the connection vanish.
\begin{thm}\label{maithm:subRep} For a $\bZ_p$-local system $\bL$ over $\frakX_k$, the following statements hold.
\begin{itemize}
	\item[i).] There is an injective morphism between Higgs bundles
	\[(E,\theta)_\bL^{\mO\bB_{\rm dR}}\rightarrow (E,\theta)_\bL.\]
	Consequently, the Higgs bundle $(E,\theta)_\bL^{\mO\bB_{\rm dR}}$ is a semistable. We view $(E,\theta)_\bL^{\mO\bB_{\rm dR}}$ as a sub-Higgs bundle in $(E,\theta)_\bL$ via this injective morphism.
	\item[ii).] According Faltings' Simpson correspondence, the sub-Higgs bundle $(E,\theta)_\bL^{\mO\bB_{\rm dR}}$ corresponds to a sub-generalized representation. Then this generalized representation is actually a genuine sub-$\mO_{\hbk}$-representation, denoted by $\bW^{\rm geo}_{\Ohbk}$, of $\bL^{\rm geo}\otimes \mO_{\hbk}$.
\end{itemize}
\end{thm} 

Before we start the proof, we need two lemmas as follows. 
\begin{lem}\label{lem:subsheaf1}
	Let $\frakX$ be a projective curve over $\Ok$ with semistable reduction. Assume $\mF$ is a sub-vector bundle of 
	\[\mG=\mO_{\frakX_\kappa}^{\oplus s}\]
	of degree $0$ over the special fiber $\frakX_\kappa$. Denote $r=\rank(\mF)\leq s$. Then 
	\[\mF\cong \mO_{\frakX_\kappa}^{\oplus r}.\]
\end{lem}

\begin{proof} Let $e_1,\cdots,e_s$ be a basis of $H^0(\frakX,\mG)$ and let $Y_1,\cdots,Y_m$ be all irreducible components of $\frakX_\kappa$. denote $e_{ij}:= e_i\mid_{Y_j}$\\[-2mm]
	
	Firstly, we show $\mG\mid_{Y_i}$ is semistable, that is $\deg(\mE)\leq 0$, for any subsheaf $\deg(\mE)$ of $\mG\mid_{Y_i}$. Suppose that $\deg(\mE)>0$. Then we consider the subbundle $\mE'$ of $\mE$ with maximal slope (this slope is obviously bigger than $0$). Since 
	\[0 \neq \mE'\subset \mG\mid_{Y_i} = \bigoplus_{i=1}^{s} \mO_{Y_j} \cdot e_{ij},\]
	there exists a nonzero projection map 
	\[\pi_i \colon \mE' \rightarrow \mO_{Y_j}\cdot e_{ij}.\]
	Considering the following exact sequence
	\[0\rightarrow \ker{\pi_i} \longrightarrow \mE' \longrightarrow \mathrm{im}(\pi_i)\rightarrow 0.\]
	Since the slope of the image of $\pi_i$ is nonpositive and the slope of $\mE'$ is positive, the slope of $\ker(\pi_i)$ is bigger than that of $\mE'$. This contradicts the definition of $\mE'$. Thus $\deg(\mE)\leq 0$. \\[-2mm]
	
	Secondly, for all $i=1\cdots,m$ one has 
	\[\deg(\mF\mid_{Y_i})=0.\]
	This is because 
	\[0=\deg(\mF) = \sum_{i=1}^m \deg(\mF\mid_{Y_i})\]
	and $\deg(\mF\mid_{Y_i})\leq 0$ from the first step. \\[-2mm]
	
	Thirdly, every degree $0$ rank $r$ subbundle $\mE\subset\mG\mid_{Y_i}$ is isomorphic to $\mO_{Y_i}^{\oplus r}$. Choose a nonzero projection 
	\[\pi_j\colon \mE\rightarrow \mO_{Y_i}\cdot e_{ij}.\]
	Since $\mG\mid_{Y_i}$ is semistable of slope zero, one obtains that $\mE$ is also semistable of slope zero. Hence $\deg(\mathrm{im}(\pi_j))\geq0$. So it is surjective. By induction on the rank, we may assume \[\ker(\pi_j)\cong\mO_{Y_i}^{\oplus (r-1)}.\]
	Then the subsheaf 
	\[\ker(\pi_j)+\mO_{Y_i}\cdot e_{ij}\subset \mG\mid_{Y_i}\]
	is a direct summand of rank $r$ and the projection mapping from $\mE$ to this summand is an isomorphism.
	\\[-2mm]
	
	Finally, we show that 
	\[\mF\cong \mO_{\frakX_\kappa}^{\oplus r}.\]
	For any $P_{ij}\in Y_i\cap Y_j$. Consider the diagram
	\begin{equation*}
		\xymatrix{
			H^0(\frakX_\kappa,\mF) \ar@{^(->}[r] \ar[d]^{\rm res} & H^0(\frakX_\kappa,\mG)\ar[d]_{\cong}^{\rm res}\\
			H^0(Y_i,\mF\mid_{Y_i}) \ar@{^(->}[r] \ar[d]_{\cong}^{\rm res} & H^0(Y_i,\mG\mid_{Y_i}) \ar[d]_{\cong}^{\rm res} \\
			\mF\mid_{P_{ij}} \ar@{^(->}[r] & \mG\mid_{P_{ij}}\\
		}
	\end{equation*}
	Since 
	\[\mG=\mO_{\frakX_\kappa}^{\oplus s},\]
	the two vertical maps on the right hand side are isomorphisms. We identify these three vector spaces on the right hand side via these two isomorphism. Then $H^0(Y_i,\mF\mid_{Y_i})$ is a $r$-dimensional sub-vector spaces of $H^0(\frakX_\kappa,\mG)$, which does not depend on the choice of the index $i$. This subspace generates a subbundle $\mF'\subset \mG$, which is clearly isomorphic to $\mO_{\frakX_\kappa}^{\oplus r}$. By the construction of $\mF'$, one has 
	\[\mF'\mid_{Y_i}=\mF\mid_{Y_i}\]
	for each $i$. Thus $\mF'=\mF$.
\end{proof}

\begin{lem}\label{lem53}
	Let $\frakX$ be a projective curve over $\Ok$ with semistable reduction. Assume $\mF$ is a sub-vector bundle of 
	\[\mG=(\mO_{\frakX}/p^n)^{\oplus s}\]
	of degree $0$ over $\frakX\otimes_W W_n$. Denote $r=\rank(\mF)\leq s$. Then 
	\[\mF\cong(\mO_{\frakX}/p^n)^{\oplus r}.\]
\end{lem}
\begin{proof}
	By Lemma~\ref{lem:subsheaf1}, there exists a basis $e_1,\cdots,e_s$ of $\mG$ such that 
	\[\mF/p = \mF'/p\]
	where $\mF'$ is the direct summand of $\mG$ generated by $e_1,\cdots.e_r$. Consider the composition
	\[f\colon \mF \rightarrow \mG \twoheadrightarrow \mF'\cong(\mO_{\frakX}/p^n)^{\oplus r}.\]
	Since 
	\[\mF/p = \mF'/p,\]
	one obtains that 
	\[f\pmod{p} = \mathrm{id}.\]
	Hence $f$ is an isomorphism.
\end{proof}
\begin{proof}[Proof of Theorem~\ref{maithm:subRep}]
	For the first part, let us recall 
	\[\RH^{Sch}(\bM)=\nu_*(\bL\otimes_{\hat{\bZ}_p} \OBdR)\]
	is filtered de Rham $\mO_\Xkadet$-module over $\Xkadet$. It is algebraic according~\cite[Theorem 9.1]{Sch13} or K\"opf's thesis~\cite{Kopf}. Denote
	\[\bM^{\OBdR}:= \mM^{Sch}\circ \RH^{Sch}(\bM)\]
	which is naturally embedding into the $\bBdRp$-representation 
	\[\bM=\bL\otimes_{\hat{\bZ}_p} \bBdRp.\] 
	\add{This is because the functor is $\RH^{Sch}$ is left exact and the $\bBdRp$-local systems associated to filtered $\mO_{\Xkadet}$-modules with integrable connections are closed under subs and quotients. Denote by $K$ the kernel of $\bM^{\OBdR}\rightarrow\bM$. Then taking $\RH^{Sch}$, then $\RH^{Sch}(K)$ is the kernel of an identity map and hence $\RH^{Sch}(K)=0$. So $K=0$.}
	Taking the first grading piece, one has sub-generalized representation
	\begin{equation}\label{equ:subrep}
		\bM^{\OBdR}/\xi\bM^{\OBdR} \subset \bM/\xi\bM=\bL\otimes \hOX
	\end{equation} 
	By Theorem~\ref{thm:main2}, one has 
	\[\bM^{\OBdR}/\xi\bM^{\OBdR}=\bD^{Fal}((E,\theta)_\bL^{\OBdR})\]
	where $(E,\theta)_\bL^{\OBdR}:=\Gr\circ\RH^{Sch}(\bM)$. Hence 
	\[\mH^{Fal}(\bM^{\OBdR}/\xi\bM^{\OBdR}) = (E,\theta)_\bL^{\OBdR}.\]
	Recall $(E,\theta)_\bL:=\mH^{Fal}(\bL\otimes \hat{\mO}_{\Xkad})$. Taking the functor $\mH^{Fal}$ on~\eqref{equ:subrep}, one gets an injective morphism 
	\[(E,\theta)_\bL^{\OBdR} \rightarrow (E,\theta)_\bL.\]
	\\[3mm]
	For the second part, as notation in the proof of the first part, the generalized representation corresponding the sub-Higgs bundle $(E,\theta)_\bL^{\mO\bB_{\rm dR}}$ is $\bM^{\OBdR}/\xi\bM^{\OBdR}$. We nee to Show this generalized representation comes from a genuine $\hbk$-geometrical representation. By Lefschetz hyperplane section theorem, we assume that $\frakX_{\hbk}$ is a curves, and after a base change on $\frakX$ \'etale over $\frakX$, we assume the restriction $\bL^{\rm geo}$ of $\bL$ modulo $p^2$ on $\Xhbkadet$ is trivial. Since $(E,\theta)_\bL^{\OBdR}$ over $\frakX_{\hbk}$ is the graded Higgs bundle of the de Rham bundle
	$(\mE_{\text{\'et}},\nabla,\Fil)$ it is of degree $0$. Since
	$(E,\theta)_\bL^{\OBdR}\subset (E,\theta)_\bL$ and $(E,\theta)_\bL$ is of degree zero and semistable, one obtains that $(E,\theta)_\bL^{\OBdR}$ is also semistable. \\
	On the other hand, the small representation $\bL^{\rm geo}$ corresponds to an extension $(\mE,\theta)_\bL$ over $\mO_{\hbk}$ of $(E,\theta)_\bL$ under Faltings' integral Simpson correspondence.
	Furthermore, for any $n\in \bN$, we may take a morphism 
	\[\pi: Y\to \frakX\]
	\'etale over $\frakX$ to make $\pi^*\bL^{\rm geo}$ is trivial modulo $p^{n+1}$. We may extend $\pi$ to a morphism 
	\[\pi: \frakY\to \frakX\]
	by choosing a suitable semistable module $\frakY$ of $Y$. Fixing $A_2(\Ok)$-lift of $\frakX$ and $\frakY$ under Faltings' integral correspondence $\pi^*\bL^{\rm geo}$ corresponds to the twisted pulled $\pi^\circ(E,\theta)_{\bL^{\rm geo}}$, which is the trivial Higgs bundle modulo $p^{n}$. The extension 
	\[(E,\theta)_\bL^{\OBdR}\subset (E,\theta)_\bL\]
	is also small which corresponds to a small generalized subrepresentation $\bW\subset \bL^{\rm geo}$. The pullback \[\pi^*(\bW)\subset \pi^*\bL^{\rm geo}\]
	corresponds to the twisted pull-back 
	\[\pi^\circ(E,\theta)_\bL^{\OBdR}\subset \pi^\circ(E,\theta).\]
	Note that $\pi^\circ(E,\theta)$
	is the trivial Higgs bundle over $Y_{\mO_{\hbk}/p^{n}\mO_{\hbk}}$ and the degree of the sub Higgs bundle $\pi^\circ(E,\theta)_\bL^{\OBdR}$ over $Y_{\mO_{\hbk}/p^{n}\mO_{\hbk}}$ equals zero. Thus, by Lemma~\ref{lem53}, the sub-Higgs bundle $\pi^\circ(E,\theta)_\bL^{\OBdR}$ is a trivial of rank-$r$.
	This implies that $\pi^*\bW$ (mod $p^n$) is an $\mO_{\hbk}/p^s\mO_{\hbk}$-geometrical representation.
	Note that the category of $\mO_{\hbk}$-geometrical representations is a fully faithful subcategory of the category of generalized representations. Thus $\pi^*(\bW)$ descends back to the original $\bW$, whose modulo $p^s$ reduction is an $ \mO_{\hbk}/p^n\mO_{\hbk}$-geometrical representation.
	Taking inverse limits, one gets that $\bW$ is an actual $\mO_{\hbk}/p^n\mO_{\hbk}$-representation of the geometric \'etale fundamental group, see a similar argument in~\cite{DW} for descending representations arising from strongly semistable vector bundles.
\end{proof} 

\begin{corollary}
	If $\bL$ is geometrically absolutely irreducible and $(E,\theta)_{\bL}^{\mO\bB_{\rm dR}}$ is nontrivial,
	then $\bL$ is a de Rham representation.
\end{corollary}

In general we ask the following conjecture:
\begin{conjecture}[Descending]
	The geometric sub-local system in Theorem~\ref{thm:subRep} descends to a $\bZ_p$-\'etale subrepresentation
	\[ \bW\subset \bL. \]
	Consequently, one has that $\bW$ is the maximal de Rham subrepresentation of $\bL$.
\end{conjecture}

\section{ An analogue of $\bC^*$-action} \label{sec:C^*}

In complex nonabelian Hodge theory, one defines a $\bC^*$-action on the moduli space of Higgs bundles $(E,\theta)$ in a natural way by taking multiplication on the Higgs field. By Simpson, a Higgs bundle is graded if and only if its isomorphic class is invariant under this $\bC^*$-action. The induced action on the moduli space of local systems via Simpson correspondence reminds very mysterious. In this section, the Proposition~\ref{why analogue} implies that the $p$-adic analogue of Simpson's $\bC^*$-action should be defined to be the action of Galois on the category of Higgs bundles.\\[0cm]

Let us firstly introduce a natural action of Galois group on generalized representations and Higgs bundles. 
\begin{defi}\label{def:GalGRep}
	Let $\delta\in \Gal(\bark/k)$ be an element in the absolute Galois group of $k$. Fix a lift $\widehat{\delta}\in \pi_1(X)$ which lifts $\delta$. 
	\begin{itemize}
		\item[i).] Let $(\bV,\rho)$ be a generalized representations. We define a new generalized representation $(\bV,\rho)^{\delta}$ by
		\[\bV^{\delta}:= \bV\otimes_{\widehat{\delta}}\OXp\]
		and 
		\[\rho^{\delta}(\sigma):=\rho(\widehat{\delta}^{-1}\circ \sigma \circ \widehat{\delta}) \otimes_{\widehat{\delta}}\mathrm{id}.\]
		\item[ii).] Let $(E,\theta)$ be a Higgs bundle. we define a new Higgs bundle $(E,\theta)^{\delta}$ by
		\[E^{\delta}:= E\otimes_{\delta} \mO_{\XbarV}\]
		and
		\[\theta^{\delta}(e\otimes_\delta 1)= \theta(e)\otimes_{\delta}1\] 
	\end{itemize}
\end{defi}

\begin{rmk}
	The isomorphism class of $(\bV^\delta, \rho^{\delta})$ does not depend on the choice of $\widehat{\delta}$. Definition does given an action of the absolute Galois group on the set of isomorphic classes of generalized representations and that of Higgs bundles. The proof is routine, we leave it for the readers.
\end{rmk}

In the following, we show that these two actions are compatible under the modified Faltings' functor in Theorem~\ref{thm:modifiedFunctor}.
\begin{thm}\label{prop:actionsCompatible}	The modified Faltings' functor commutes with the Galois actions. More precisely, for a very small Higgs bundle $(E,\theta)$ 
	\[\Big( \bD^{Fal}(E,\theta) \Big)^\delta \cong \bD^{Fal}\Big( (E,\theta)^\delta \Big).\]
\end{thm}
\begin{proof} We only need to check this locally over a small affine subset $\frakU=\Spec(\frakR)$. Recall from \eqref{underlyingMod} that $\bV = E\otimes \hbR$. Due to the following commutative diagram
\begin{equation}
\xymatrix{
\RbarV \ar[r] \ar[d]^{\delta} & \hbR \ar[d]^{\widehat{\delta}}\\
\RbarV \ar[r] & \hbR\\
}
\end{equation}
there exists a canonical isomorphism
\[E^\delta\otimes \hbR \rightarrow \bV \otimes_{\widehat{\delta}} \hbR\]
sending $(e\otimes_{\delta}1)\otimes1$ to $(e\otimes1)\otimes_{\widehat{\delta}} 1$. In the following, we identify these two module via this isomorphism. Choose local parameters $\TT$ of $\frakR$ over $\Ok$ and fix a local lift $\tau$. Denote by $(\bV',\rho')$ the generalized representation attached to $(E^{\delta},\theta^\delta)$ with respect to the lift $\widehat{\delta}\circ\tau$. Then
\begin{equation*}
\begin{split}
\rho'(\sigma)((e\otimes_{\delta}1)\otimes 1) :=
& \sum_{I} \frac{\big(\theta^{\delta}(\xi\partial)\big)^I(e\otimes_\delta 1)}{I!}\otimes \left(\frac{\tau(T)^{\widehat{\delta}}-(\tau(T)^{\widehat{\delta}})^{\sigma}}{\xi}\right)^I\\
&= \sum_{I} \frac{\big(\theta(\xi\partial)\big)^I(e)}{I!}\otimes_\delta \left(\frac{\xi}{\xi^\delta}\right)^I \otimes \left(\frac{\tau(T)^{\widehat{\delta}}-(\tau(T)^{\widehat{\delta}})^{\sigma}}{\xi}\right)^I\\
&= \sum_{I} \frac{\big(\theta(\xi\partial)\big)^I(e)}{I!} \otimes \left(\frac{\tau(T)-\tau(T)^{\widehat{\delta}^{-1}\circ\sigma \circ \widehat{\delta}}}{\xi}\right)^I \otimes_{\widehat{\delta}} 1\\
&= \rho(\widehat{\delta}^{-1}\circ\sigma \circ \widehat{\delta})(e\otimes1) \otimes_{\widehat{\delta}} 1\\
& = \rho^\delta(\sigma)((e\otimes1)\otimes_{\widehat{\delta}} 1)\\
\end{split}
\end{equation*}
Thus $\rho'$ coincides with $\rho^{\delta}$ and the theorem follows.
\end{proof}

\begin{remark}
According the theory of Higgs-de Rham flow built in~\cite{LSZ1}, there is an equivalent functor
\begin{center}
	$\bD_{crys} \colon $ \{periodic Higgs bundles\} $\rightarrow$ \{crystalline representations\}.
\end{center}
Example in~\cite[section $5$]{Fal05} claims that his functor from the category $\MF$ of Fontaine-Faltings modules to the category of crystalline representations is compatible with his $p$-adic Simpson correspondence. This example implies that the modified functor $\bD^{Fal}$ coincides with $\bD_{crys}$ after twisting the Higgs field by $t^{-1}$ (we note that it is the same as twisting by $\xi^{-1}$, because the Higgs bundle is graded). That is for any period Higgs bundle
\[\bD_{crys}(E,\theta) \cong \bD^{Fal}(E,\theta t^{-1}).\]
Since periodic Higgs bundle is always defined to be graded. the functor $\bD_{crys}$ also communicates with the Galois action by Theorem~\ref{prop:actionsCompatible}.	This generalizes~\cite[Corollary 4.3]{ST}, where Sheng and Tong showed the commutativity for a periodic Higgs bundles $(E,0)$ with zero Higgs field.
\end{remark}

\begin{example}\label{exm:basechangehiggs}
Let $(E_0,\theta_0)$ be a very small usual Higgs bundle over $\frakX$, that is the Higgs field $\theta_0$ is contained in $\End(E_0)\otimes\Omega_\frakX^1$ with $\theta_0$ divided by a sufficiently large power of $p$. Let $(\bV,\rho)$ be the generalized representation corresponding to the very small Higgs bundle $(E,\theta)$, where $E=E_0\otimes \mO_{\XbarV}$, $\theta=\theta_0/t$, and $t=\log([(\zeta_{p^n})_n])\in \BdRp/(\xi^2)$. Then for any unit $c\in \bZ_p^\times$ sufficient closed to $1$, then there exists $\delta\in \Gal(\bark/k)$ such that 
\[\delta(t) = c^{-1}t.\]
Hence
\[(E,\theta)^{\delta} \cong (E,c\theta).\]
By Theorem~\ref{prop:actionsCompatible}, the Higgs bundle $(E,c\theta)$ is corresponding to the generalized representation $(\bV,\rho)^{\delta}$.
\end{example}

\begin{prop}\label{why analogue}
Let $(E_0,\theta_0)$, $(E,\theta)$ and $(\bV,\rho)$ be given as in Example~\ref{exm:basechangehiggs}. Then $(E,\theta)$ is graded if and only if the isomorphic class of $(\bV,\rho)$ is invariant under the action of $\Gal(\bark/k)$.
\end{prop}
\begin{proof} Suppose that the isomorphic class of $(\bV,\rho)$ is invariant under the $\Gal(\bark/k)$-action. Then there exists an isomorphism 
	\[(E,c\theta)\cong (E,\theta)\]
for some $c\in\bZ_p$ sufficient closed to $1$. By Simpson's method, the Higgs bundle $(E,\theta)$ is graded. Conversely, Since $\theta$ is graded 
\[(E,c\theta)\cong (E,\theta)\]
for any $c\in \bZ_p^\times$. In particular 
\[(E^\delta,\theta^\delta)\cong (E,\theta).\]
Thus by Proposition~\ref{prop:actionsCompatible}, isomorphic class of $(\bV,\rho)$ is invariant under the action of $\Gal(\bark/k)$.
\end{proof}

\begin{prop}\label{nil}
Let $(E,\theta)$ be a Higgs bundle over $\XhbV$ such that the isomorphism class of its associated generalized representation is fixed by the action of $\Gal(\bark/k)$. Then the Higgs field $\theta$ is nilpotent.
\end{prop}

\begin{proof} Since the underlying vector bundle $E$ is invariant under the Galois group action, it is defined over $\frakX$, that is there exists a vector bundle $E_0$ over $\mX$ such that 
	\[E = E_0\otimes \mO_\XbarV.\]
Under this identity, the Higgs field $\theta$ is an element in
\[ \End_{\mO_\frakX}(E_0)\otimes \Omega_{\frakX}^1\otimes \mO_{\hbk}\cdot \xi^{-1}.\]
Since $(E,\theta)$ is invariant under the Galois action, the coefficients of $\det(T\cdot \mathrm{id}-\theta)$ are all invariant under the Galois action. But the coefficients are contained in 
\[\Gamma(\frakX,\Omega^{\otimes i}_{\frakX/\Ok}) \otimes \widehat{\overline{k}}(-i), \text{ where } i=0,\cdots,r.\] 
Since, for any $i=1,\cdots,r$,
\[\left(\Gamma(\frakX,\Omega^{\otimes i}_{\frakX/\Ok}) \otimes \widehat{\overline{k}}(-i)\right)^{\Gal(\bark/k)} = 0,\]
one has 
\[\det(T\cdot \mathrm{id}-\theta) = T^r.\]
Thus $\theta$ is nilpotent.
\end{proof}

\begin{prop}\label{prop:rank2GalInvimpliesgrad} Let $(E,\theta)$ be a rank $2$ Higgs bundle. If its isomorphism class is invariant under the action of $\Gal(\bark/k)$, then $(E,\theta)$ is graded.
\end{prop}

\begin{proof}
By Lefschetz hyperplane section theorem, we may assume that the relative dimension of $\frakX/\Ok$ is $1$. Denote \[L=\ker\theta.\]
If $\rank(L)=2$, there is nothing to prove. Now we assume $\rank(L)=1$ and denote $L'=E/L$. Then the Higgs field is equivalent to
\[\theta\colon L' \rightarrow L \otimes \Omega^1_{\frakX/\Ok}.\]
Since 
\[(E^\delta,\theta^\delta) \cong (E,\theta),\]
there exists a commutative diagram as follows
\[\xymatrix{
L' \ar[r]^-{\theta} \ar[d]^{\cong}_{f_{\delta}} & L \otimes \Omega^1_{\frakX/\Ok}\cdot t^{-1} \ar[d]^{\cong}_{g_\delta}\\
L'^{\delta} \ar[r]^-{\theta^\delta} & L^\delta \otimes \Omega^1_{\frakX/\Ok} \cdot t^{-1} \\}\]
Hence the line bundles $L$ and $L'$ are defined over $\frakX$. That is, there exist line bundles $L_0$ and $L_0'$ over $\frakX$ such that 
\[L' = L'_0\otimes \mO_{\frakX_{\mO_{\hbk}}}\]
and 
\[L = L_0\otimes \mO_{\frakX_{\mO_{\hbk}}}.\]
The Higgs field can be rewrite as
\[\theta = \sum_{i=1}^m\omega_i \otimes c_i \cdot t^{-1} \in \Hom(L_0',L_0\otimes\Omega^1_{\frakX/\Ok}) \otimes \mO_{\hbk}\cdot t^{-1}\]
where $\omega_1,\cdots,\omega_m$ is a basis of $\Hom(L_0',L_0\otimes\Omega^1_{\frakX/\Ok})$ and $c_i\in \mO_{\hbk}\setminus{0}$. The existence of an isomorphism 
\[(E^\delta,\theta^\delta) \cong (E,\theta)\]
implies there exists a nonzero $\lambda=\lambda_\delta\in \mO_{\hbk}^\times$ such that
\[\delta(c_i) = \lambda c_i,\quad \text{for all } i=1,2,\cdots,m.\]
Denote $c:=c_{i_0}\neq 0$ for some $1\leq i_0\leq m$. Then $\frac{c_i}{c}$ is invariant under the action of $\Gal(\bark/k)$. Hence $\frac{c_i}{c}$ is contained in $k$. Denote 
\[\omega= \sum_{i=1}^m\omega_i \cdot \frac{c_i}{c}.\]
Then one has
\[\theta = \omega \otimes c \cdot t^{-1}.\]
So one has $\theta^\delta= \frac{\delta(ct^{-1})}{ct^{-1}}\cdot\theta$ and
\[(E^\delta,\theta^\delta) = (E,\frac{\delta(ct^{-1})}{ct^{-1}}\cdot\theta).\]
Since 
\[\left(\hbk\cdot t^{-1}\right)^{\Gal(\bark/k)}=0,\]
there exists $\delta\in \Gal(\bark/k)$ such that $\mu:=\frac{\delta(ct^{-1})}{ct^{-1}}\neq 1$. Then
\[(E,\mu\theta)\cong (E,\theta)\]
By the argument due to Simpson, one has $(E,\theta)$ is graded.
\end{proof}

\begin{theorem}\label{thm:rk2}
 The following statements hold.
\begin{itemize}
	\item[i).] A generalized representation corresponding to a graded Higgs bundle over $k$ is $\Gal(\bark/k)$-invariant and conversely, a Higgs bundle corresponding to a Galois-invariant generalized representation is nilpotent.
	\item[ii).] A rank-$2$ generalized representation is $\Gal(\bark/k)$-invariant if and only if the corresponding Higgs bundle is graded and defined over $k$ up to an isomorphism. 
\end{itemize}
\end{theorem}
\begin{proof}
	The first term is a sum-up of Proposition~\ref{why analogue} and Proposition~\ref{nil}. We only need to prove the second term at here. By Proposition~\ref{prop:rank2GalInvimpliesgrad} the Higgs bundle associated to a Galois invariant generalized representation is graded. According Theorem~\ref{prop:actionsCompatible}, the isomorphic classes of the Higgs bundle is also Galois invariant. 
	Conversely, if a Higgs bundle is graded and defined over $k$ then it isomorphic of a Higgs bundle of form as in Example~\ref{exm:basechangehiggs}. Proposition~\ref{why analogue} implies that it is invariant under the action of the Galois group.
\end{proof}

\begin{conjecture} \label{conj:galHiggs}
	A generalized representation is $\Gal(\bark/k)$-invariant if and only if the corresponding Higgs bundle is graded and defined over $k$ up to an isomorphism.
\end{conjecture}

\end{document}